\documentclass[reqno,12pt]{amsart}
\usepackage{amsmath, latexsym, amsfonts, amssymb}
\usepackage{mathrsfs,enumerate}
\usepackage{graphics,epsfig,color}

\numberwithin{equation}{section}

\newtheorem{theorem}{Theorem}[section]
\newtheorem{lemma}[theorem]{Lemma}
\newtheorem{prop}[theorem]{Proposition}

\newtheorem{definition}[theorem]{Definition}

\newtheorem{remark}[theorem]{Remark}
\newcommand{\eps}{\varepsilon}

\newcommand{\cP}{{\ensuremath{\mathcal P}} }

\newcommand{\bbE}{{\ensuremath{\mathbb E}} }
\newcommand{\E}{{\ensuremath{\mathbb E}} }

\newcommand{\N}{{\ensuremath{\mathbb N}} }

\newcommand{\bbP}{{\ensuremath{\mathbb P}} }

\newcommand{\R}{{\ensuremath{\mathbb R}} }

\newcommand\I{\operatorname{I}}

\renewcommand\H{\operatorname{H}}

\newcommand{\mc}[1]{{\mathcal #1}}

\newcommand{\mb}[1]{{\mathbf #1}}
\newcommand{\bb}[1]{{\mathbb #1}}
\newfont{\indic}{bbmss12}
\def\un#1{\hbox{{\indic 1}$_{#1}$}}
\title[Large deviations for renewal processes]
      {Large deviations for renewal processes}

\author[R.\ Lefevere]{Rapha\"el Lefevere}
\address{Laboratoire de Probabilit\'es
 et Mod\`eles Al\'eatoires (CNRS UMR 7599), Universit\'e Paris 7
 -- Denis Diderot, UFR Math\'ematiques, Case 7012, B\^atiment
 Chevaleret, 75205 Paris Cedex 13, France}
\email{lefevere\@@math.jussieu.fr}

\author[M.\ Mariani]{Mauro Mariani}
\address{Laboratoire d'Analyse,
 Topologie, Probabilit\'es (CNRS UMR 6632), Universit\'e Aix-Marseille,
 Facult\'e des Sciences et Techniques
 Saint-J\'er\^ome, Avenue Escadrille Normandie-Niemen 13397 Marseille
 Cedex 20, France}
\email{mariani\@@cmi.univ-mrs.fr}

\author[L. Zambotti]{Lorenzo Zambotti}
\address{Laboratoire de Probabilit{\'e}s
 et Mod\`eles Al\'eatoires (CNRS UMR. 7599)
 Universit\'e Paris 6 -- Pierre et Marie Curie,
 U.F.R. Math\'ematiques, Case 188, 4 place
 Jussieu, 75252 Paris cedex 05, France }
\email{lorenzo.zambotti\@@upmc.fr}

\begin{document}

\begin{abstract}
  We investigate large deviations for the empirical measure of the
  forward and backward recurrence time processes associated with a classical
  renewal process with arbitrary waiting-time distribution. The Donsker-Varadhan theory cannot be applied in this case, and indeed it turns 
out that the large deviations rate functional differs from the one suggested by such a theory. In particular, a non-strictly convex and non-analytic rate functional is obtained. 
\end{abstract}



\maketitle

\section{Introduction}
\subsection{Motivations from Statistical Physics}
In large deviations theory, the appearance of rate functionals with singular points (that is, points of non-differentiabi\-lity or non-analiticity) is a 
feature marking the existence of critical phenomena in the underlying stochastic processes. Existence of such singularities is of particular interest in a number of situations, for instance  whenever these functionals are associated with deviations of physical quantities in Statistical Mechanics models, as they identify phase transitions. Moreover, values of the parameters in which deviations functionals are convex, or affine, or non-convex are related to different behaviors of the system.

In this respect, this work has been initially motivated by the appearance of affine stretches in large deviations rate 
functionals of Statistical Mechanics models, whose dynamics depends on renewal processes. In \cite{LMZ1} a heat conduction model is introduced, and it is shown that the rate functional of the energy current is convex but not strictly convex, with an affine behavior over two distinct intervals, from which the  appearance of critical points. In these conditions, the classical G\"artner-Ellis Theorem does not yield the full large deviations principle and a more detailed understanding of the random dynamics is necessary. 

In this paper we do not pursue this Statistical Mechanics interpretation, but rather show how affine stretches in large 
deviations rate functionals of renewal processes arise when the inter-arrival times
have heavy tails. We argue that in such situations the Donsker-Varadhan approach \cite{DV} does not yield a good rate functional and therefore the classical framework must be modified.

Before detailing the main result, we recall an example concerning large deviations of the renewal cumulative process, with the aim to 
underline that our Theorem~\ref{t:ld1} below may have interesting consequences not related to Statistical Mechanics.

\subsection{A motivating example}  Suppose a sequence of tasks $i=1,\,2\ldots$ is given, where the task $i$ takes a \emph{service time} 
$\tau_i$ to be accomplished. If the reward paid for executing such a task $i$ is function $F(\tau_i)$ of the time elapsed to accomplish it, then 
the total amount $C_t$ gained at time $t>0$ is
\begin{equation}
\label{Ct}
C_t:=\sum_{i=1}^{N_t-1} F(\tau_i), \qquad t>0,
\end{equation}
where 
\begin{equation*}
N_t:=\inf \left\{ n\geq 0\,:\: \sum_{i=1}^n \tau_i >t \right\},
\end{equation*}
and $C_t=0$ if $N_t=1$. When the service times $\tau_i$ are random, the study of the {\it cumulative} process $(C_t)_{t\geq 0}$ can be of 
interest in many applications, for instance queueing and risk theory. 

We assume throughout the paper that  the sequence $(\tau_i)_{i\ge 1}$ is an i.i.d.\ sequence of positive random variables and that $F \colon ]0,+\infty[\mapsto[0,+\infty[$ is bounded and continuous. The law of $\tau_i$ is an arbitrary probability measure $\psi$ on $]0,+\infty[$, without any 
moment assumption. Then $N_t$ is a so called renewal counting process and it is easily seen that a.s.\
\begin{equation*}
\lim_{t\to+\infty} \frac{C_t}t = \lim_{t\to+\infty}  \frac{N_t-1}t \,  \frac1{N_t-1}
\sum_{i=1}^{N_t-1} F(\tau_i) = \frac{\E(F(\tau_1))}{\E(\tau_1)}\in[0,+\infty[.
\end{equation*}
This is therefore the total cost per unit of time on a large time interval. A natural question, especially in the interpretations provided above, is 
the study of large deviations for the mean payoff $C_t/t$ as $t \to +\infty$.

Define $\Lambda^* \colon [0,+\infty[^2\, \mapsto [0,+\infty]$, the
Legendre transform of the map $\Lambda(x,y) :=\log \psi(e^{x \tau + yF})$, namely
\begin{equation}
\label{Lambda_F}
\Lambda^*(a,b) := \sup_{x,y\in\R} \big\{ ax+by -
      \log \psi\left(e^{x \tau + y F}\right) \big\}, \qquad a,b\geq 0.
\end{equation}
 A first result obtained as a consequence of the theory developed below is
\begin{theorem}
\label{t:ldcounting}
%Assume that $F \in C(]0,+\infty[)$ is bounded. 
The law of the random variable $C_t/t$ 
defined by \eqref{Ct} satisfies a large deviations principle with good rate functional $J_{F}\colon
 [0,+\infty[ \to [0,+\infty]$ defined as
 \begin{equation}
 \label{e:IN}
 J_F(m):= \inf\left\{ \beta \, \Lambda^*(1/\beta,m/\beta)\,:\:
 \, \beta>0\right\}
\end{equation}
i.e.\ for each closed set $\mc C \subset \R$
\begin{equation*}
\varlimsup_{t\to +\infty} \frac1t \, \log \bbP\left(C_t/t\in \mc C\right)
  \le - \inf_{\mc C} J_F
\end{equation*}
and for each open set $\mc O \subset \R$
\begin{equation*}
\varliminf_{t\to +\infty} \frac1t \, \log \bbP\left(C_t/t\in \mc O\right)
  \ge -\inf_{\mc O} J_F.
\end{equation*}
\end{theorem}
This result is known for a broader class of cumulative processes, but in the contest of a {\it bounded} sequence $(\tau_i)_{i\ge 1}$, see 
\cite{dm1,russ}, or in more generality for $F\equiv 1$, corresponding to the large deviations of $N_t/t$, see \cite{GW}. Here we address the case where $\tau_i$ has an arbitrary distribution, and
indeed large deviations display a more interesting behavior in the case of heavy tailed distribution of $\tau_i$, as explained in section \ref{affine} below.

%If $F\equiv 1$, then $C_t=N_t-1$, the counting process, and a large deviation principle for $N_t/t$ has been proved in \cite{GW}, with rate 
%$J_1(m)=m\Lambda^*(1/m)$, where $\Lambda^*(a):=\sup_x(ax-\log \psi(e^{x \tau}))$. This result also follows from 
%Theorem~ref{t:ldcounting} 
%for $F=1$, since $\Lambda(x,y) :=y+\log \psi(e^{x \tau})$ for $F\equiv 1$ and
%\begin{equation*}
%\begin{split}
%J_1(m) & =\inf_{\beta>0}\sup_{x,y}\left\{ x+y(m-\beta) -
%      \beta\log \psi\left(e^{x \tau}\right)\right\}
%=
%%m\sup_{x}\left\{ x/m -
% %     \log \psi\left(e^{x \tau}\right) \, \right\}\\ & =
%m\Lambda^*(1/m).
%      \end{split}
%\end{equation*}

\subsection{Empirical measures}
We refer to \cite{demzei} for general large deviations theory and \cite{asmussen} for renewal processes.
%Some classical Markov processes have been introduced in the above setting. 
We denote the classical renewal process associated with 
$(\tau_i)_{i\geq 1}$ by
\begin{equation*}
S_0:=0, \qquad S_n:=\tau_1+\cdots+\tau_n, \qquad n\geq 1,
\end{equation*}
so that the number of renewals before time $t>0$ is also written as
\begin{equation*}
N_t:=\sum_{n=0}^\infty \un{(S_n\leq t)}=
\inf\left\{ n\geq 0\,:\: S_{n} >t\right\}.
\end{equation*}
Recall that the {\it backward recurrence time process} $(A_t)_{t\geq 0}$ and
the {\it forward recurrence time process} $(B_t)_{t\geq 0}$ are defined by
\begin{equation*}
A_t:=t-S_{N_t-1}, \qquad B_t:=S_{N_t}-t, \qquad t\geq 0.
\end{equation*}
%Notice that $A_t\in [0,+\infty[$ and $B_t\in\, ]0,+\infty[$ for all $t\geq 0$.

It is well known and easy to prove that the process $(A_t,B_t)_{t\geq 0}$ is Markov. One can consider its empirical measure 
\begin{equation}\label{mu_ti}
\mu_t:= \frac 1t \int_{[0,t[}\delta_{(A_s,B_s)}\, ds\in {\mc P}(]0,+
\infty[^2),
\end{equation}
i.e.\ for all bounded continuous $f\colon ]0,+\infty[^2 \to \R$
\begin{equation*}
\mu_t(f):= \frac 1t \int_{[0,t[}f(A_s,B_s)\, ds.
\end{equation*}
The Donsker-Varadhan (DV) theory \cite{DV}, \cite[Chap.~6]{demzei}, provides a general result for the large deviations of the empirical 
measure of Markov processes on metric spaces. However, the standard assumptions of classical DV theorems do not hold here. In fact, even formally, the DV rate functional does not provide the right large deviations functional, see Section~\ref{s:i5} below for a discussion. 

The main result of this paper, in Theorem~\ref{t:ld1} below, is a large deviations principle for the law ${\bf P}_t$ of $\mu_t$ %in ${\mc P}(]0,+\infty]^2)$ 
as $t\to+\infty$ with an explicit rate functional $\I$ defined in \eqref{e:I2}. This allows to deduce Theorem~\ref{t:ldcounting} with a 
contraction principle and obtain a relationship between $\I$ and $J_F$, see \eqref{IJ_F} below.

\subsection{The large deviations rate functional}
In order to properly define the rate functional $\I$ for the large
deviations of the law $\mb P_t$ of $\mu_t$, some preliminary notation is needed.

%The i.i.d.\ sequence $(\tau_i)_{i\geq 1}$ is
%defined on a probability space $(\Omega,\cF,\bbP)$,
%$\tau_i\in\,]0,+\infty[$ a.s. for all $i$ and we denote the law of $\tau_i$ by $\psi(d\tau)$.  We stress that $\psi$ is an arbitrary probability 
%measure over $]0,+\infty[$, without any moment assumption.

For a Polish space $X$, $C_\mathrm{b}(X)$ denotes the space of real bounded
continuous functions on $X$, and $\cP(X)$ denotes the Polish space of Borel probability measures on $X$, equipped with its narrow (weak) 
topology. For $\mu \in \cP (X)$ is a Borel probability measure on a metric space $X$ and $f \colon X \to [0,+\infty]$ a Borel function, the 
notation
\begin{equation*}
\mu(f):=\int_X f\, d\mu,
\end{equation*}
is used throughout the paper. We also adopt the conventions
\begin{equation*}
  \label{convention}
0\cdot \infty =0, \qquad \frac1\infty=0.
\end{equation*}
The space $]0,+\infty]$ will be endowed throughout
the paper with a metric which makes it isometric to $[0,+\infty[$, for instance by setting $t\colon ]0,+\infty] \to[0,+\infty[$,
\begin{equation*}
t(p):=\frac 1p, \qquad
d(p,p'):=\left|t(p) -t(p')\right|, \qquad p,\,p' \in  ]0,+\infty].
\end{equation*}
Thus $(]0,+\infty],d)$ is a Polish space. Let
\begin{equation*}
\tau \colon ]0,+\infty]\times  ]0,+\infty]\to \, ]0,+\infty], \qquad
\tau(a,b):=a+b,
\end{equation*}
while we understand $\tau \colon ]0,+\infty]\to]0,+\infty]$ to be the identity map. Thus for $\mu\in\cP(]0,+\infty]^2)$ and $\pi\in\cP(]0,+\infty])$ 
\begin{equation*}
\mu(1/\tau)=\mu(1/(a+b))=\int_{]0,+\infty]^2} \frac1{a+b}\, \mu(da,db),
\end{equation*}
\begin{equation}
\label{e:tildepi/tau0}
\pi(\tau)=
\int_{]0,+\infty]} \tau\, \pi(d\tau), \qquad \pi(1/\tau)=
\int_{]0,+\infty]} \frac1\tau\, \pi(d\tau).
\end{equation}
Let us define $\Delta_0\subset\mc P(]0,+\infty]^2)$ as 
\begin{equation}
  \label{e:muomega0}
\begin{split}
\Delta_0:=\Big\{\mu_0\in \mc P(]0,+\infty]^2) \, : \: & \mu_0(da,db)=  \int_{[0,1]\times]0,+\infty[} \delta_{(u\tau,(1-u)\tau)}(da,db)\, du \otimes \pi(d
\tau), 
\\ & \ \pi \in \mc P(]0,+\infty[), \ \pi(1/\tau)<+\infty \Big\}.
\end{split} 
\end{equation}
In other words, $\mu_0$ is the law of $(UP,(1-U)P)$, where $U$ and $P$ are independent, $U$ is uniform on $[0,1]$ and $P$ has law $\pi\in
\mc P(]0,+\infty[)$. We also set $\Delta\subset\mc P(]0,+\infty]^2)$
\begin{equation}  \label{e:muomega}
\begin{split}
\Delta:=\Big\{ &
  \mu= \alpha \mu_0 +(1-\alpha)\delta_{(+\infty,+\infty)}\,:
  \: \mu_0 \in \Delta_0, \, \alpha\in[0,1]  \Big\}.
\end{split}
\end{equation}
If $\mu \in \Delta$ then the writing \eqref{e:muomega} is unique up to
the trivial arbitrary choice of $\mu_0$ when $\alpha=0$.

If $\nu,\mu\in\cP(X)$ then
$\H(\nu\,|\,\mu)$ denotes the relative entropy of $\nu$ with respect to $\mu$;
this notation is used regardless of the space $X$. Finally, we set
\begin{equation}
\label{e:xiphi}
\xi:= \sup \Big\{c \in \bb R\,:\: \psi(e^{c\tau})<+\infty \Big\}
\in[0,+\infty],
\end{equation}
where we recall that $\psi$ denotes the law of $\tau_i$.  

\begin{definition}
Let $\pi\in\mc P(]0,+\infty[)$ satisfy
$\pi(1/\tau)\in\,]0,+\infty[$, and set
\begin{equation}
\label{e:tildepi}
\tilde \pi(d\tau):= \frac 1{\pi(1/\tau)}\, \frac1\tau\, \pi(d\tau)\in\mc P(]0,+\infty[).
\end{equation}
Then the functionals $\I_0,\,\I \colon \mc P(]0,+\infty]^2) \to [0,+\infty]$ are defined by
\begin{equation}
\label{e:I1.5}
\I_0(\mu):=
  \begin{cases}
\pi(1/\tau) \H\big(\tilde\pi\, \big|\, \psi\big)   
           & \text{if $\mu \in \Delta_0$ is given by \eqref{e:muomega0} }
\\
+\infty & \text{if $\mu \notin \Delta_0$},
  \end{cases}
 \end{equation}
\begin{equation}\label{e:I2}
\I(\mu):=
  \begin{cases}
\alpha\, \pi(1/\tau) \H\big(\tilde\pi\, \big|\, \psi\big)   + (1-\alpha)\, \xi
           & \text{if $\mu \in \Delta$ is given by \eqref{e:muomega0}-\eqref{e:muomega} }
\\
+\infty & \text{if $\mu \notin \Delta$}.
  \end{cases}
 \end{equation}
\end{definition}
Any $\mu \in \Delta$ can be written in the form
\eqref{e:muomega}, with the only caveat that $\pi$ is not uniquely defined if
$\alpha=0$.  
%In order to have a
%notational consistency, we set $\pi=\delta_1$ whenever $\alpha=0$.
 Notice that for $\pi$ and $\tilde\pi$ as in \eqref{e:tildepi} the following relations hold
\begin{equation}
\label{e:tildepi/tau}
\tilde \pi(\tau) = \frac 1{\pi(1/\tau)}, \qquad 
\pi(d\tau):= \frac 1{\tilde \pi(\tau)}\, \tau\, \tilde \pi(d\tau).
\end{equation}

\begin{prop}
\label{p:igood2}
The functional $\I$ is \emph{good}, namely its sublevel sets are compact. 
%The functional $\I_0$ is \emph{good} if and only if $\xi=+\infty$. 
Moreover $\,\I$ is the lower-semicon\-tinuous envelope of $\I_0$.

For all bounded and continuous $F\colon ]0,+\infty[ \to[0,+\infty[$ the functional $J_F$ defined in \eqref{e:IN} is related to $\I_0$ and $\I$ by 
the formulae
\begin{equation}
\label{I_0J_F}
 J_F(m) = \min\left\{ \I_0(\mu)\,: \: \mu\in{\mc P}(]0,+\infty]^2), \, \int_{]0,+\infty]^2}\frac{F(a+b)}{a+b}\,\mu(da,db)=m
 \right\},
\end{equation}
\begin{equation}
\label{IJ_F}
 J_F(m) = \min\left\{ \I(\mu)\,: \: \mu\in{\mc P}(]0,+\infty]^2), \, \int_{]0,+\infty]^2}\frac{F(a+b)}{a+b}\,\mu(da,db)=m
 \right\}.
\end{equation}
\end{prop}

\subsection{The large deviations principle for the empirical measure}
We give here the main result of this paper.
\begin{theorem}\label{t:ld1}
The family $({\bf P}_t)_{t>0}$ satisfies a large deviations
principle with good rate $\I$ defined by \eqref{e:I2} as $t\uparrow+\infty$ with speed $t$, i.e.\ for each closed set $\mc C \subset \mc
    P(]0,+\infty]^2)$
\begin{equation}\label{upper}
 \varlimsup_{t \to +\infty} \frac{1}{t} \log {\bf P}_t(\mc C)
  \le - \inf_{u \in \mc C} \I(u)
\end{equation}
and for each open set $\mc O \subset
    \mc P(]0,+\infty]^2)$
\begin{equation}\label{lower}
 \varliminf_{t \to +\infty} \frac{1}{t}  \log {\bf P}_t(\mc O)
  \ge - \inf_{u \in \mc O} \I(u).
\end{equation}
\end{theorem}

\subsubsection{Some comments on the rate functional $\I$}
We stress again that the probability distribution $\psi$ on $]0,+\infty[$ is completely arbitrary. However the fine properties of the associated 
renewal process depend on $\psi$, and the same is true for $\I$. Define $\bar \mu \in  \mc P(]0,+\infty]^2)$ as
\begin{equation*}
\bar \mu(da,db):= 
\begin{cases}
 \int_{[0,1]\times]0,+\infty[} \frac{\tau}{\psi(\tau)} \delta_{(u\tau,(1-u)\tau)}(da,db)\, du \otimes \psi(d\tau)
& \text{if $\psi(\tau)<+\infty$}
\\
\delta_{(+\infty,+\infty)} & \text{if $\psi(\tau)=+\infty$}
\end{cases}
\end{equation*}
It follow from our results that $\mu_t\rightharpoonup\bar \mu$ as 
$t\to+\infty$.  Then
\begin{remark} $ $
{\rm
\begin{enumerate}

\item{ If $\xi=+\infty$, i.e.\ if $\psi$ has all exponential moments, then $\I \equiv\I_0$ and $\I(\mu)=0$ iff $\mu = \bar\mu \in \mc P(]0,+\infty[^2)$.
}

\item{If $\xi <+\infty$ and $\psi(\tau)=+\infty$, then $\I\neq \I_0$, and $\I(\mu)=0$ iff $\mu=\delta_{(+\infty,+\infty)}=\bar \mu $.
}

\item{If $\xi<+\infty$ and $\psi(\tau)<+\infty$. Then $\I(\bar\mu)=0$ and thus
\begin{equation*}
\I(\alpha \bar \mu + (1-\alpha) \delta_{(+\infty,+\infty)})=(1-\alpha)\xi
\end{equation*}
Therefore in this case the functional $\I$ is \emph{not strictly convex}. Still, if $\xi>0$, $\I(\mu)=0$ iff $\mu=\bar \mu$. On the other hand, if $
\xi=0$, then $\I$ vanishes identically on the segment $\{\alpha \bar \mu + (1-\alpha) \delta_{(+\infty,+\infty)},\,\alpha\in[0,1]\} $. Therefore the 
large deviations at speed $t$ do not yield the full large deviations behavior if $\xi=0$, and we shall study large deviations at a slower speed  
in a future work.
}
\end{enumerate}
}
\end{remark}
For the reader interested in the relation with the Statistical Mechanics models \cite{LZ,LMZ1} already cited above, we point out that such 
Gaussian models are related to the case (3) with $\xi=0$, so that the non-exponential decay of slow currents there observed is a 
consequence of the fact that $\I^{-1}(\{0\})$ is a whole segment in this case. We refer to \cite[section 3]{LMZ1} for further details.

%We also point out that the vanishing of $\I$ over a non-trivial interval in the case $\xi=0$ and $\psi(\tau)<+\infty$ explains a phenomenon 
%recently observed in \cite{LZ, LMZ1}, where large deviations of the current in a model for the conduction of heat has been studied; in 
%particular, a non-strictly convex rate functional arises with
% a flat region, which can be interpreted as the result of a contraction of the functional 
%$\I$ above in a case where $\xi=0$. We refer to \cite[section 3]{LMZ1} for further details.

\subsection{Relation with Donsker-Varadhan approach}
\label{s:i5}

%An attentive reader might wonder why the problem we wrote $]0,+\infty]^2$ above
%instead of $]0,+\infty[^2$, the natural state space of $(A_t,B_t)_{t\geq 0}$. In fact, the functional $\I_0$ on ${\mc P}(]0,+
%\infty[^2)$ defined in \eqref{e:I1.5}, is such that for all $m\geq 0$
%\begin{equation}\label{I_0J_F}
% J_F(m) = \inf\left\{ \I_0(\mu) \colon \mu\in{\mc P}(]0,+\infty[^2), \, \int_{]0,+\infty[^2}\frac{F(a+b)}{a+b}\,\mu(da,db)=m \right\}.
%\end{equation}
%However, it turns out that $\I_0$ is in general {\it not} a good rate functional on ${\mc P}(]0,+\infty[^2)$ and it is good if and only if all 
%exponential moments of $\tau_1$ are finite. As long as one
%exponential moment of $\tau_1$ is infinite, then the sublevels of $\I_0$ are not compact, and the law ${\bf P}_t$ of $\mu_t$ as $t\to+\infty$ 
%does not satisfy a full large deviations principle on ${\mc P}(]0,+\infty[^2)$. Notice that $C_t/t$
% does satisfy a large deviations principle in any 
%case for bounded continuous $F$. 

In the case of heavy-tailed distribution of $\tau_i$, the DV theory would yield  $\I_0$, defined in \eqref{e:I1.5}, as rate functional, while Theorem \ref{t:ld1} shows that $\I$ is the correct functional. In fact,
if $\xi<+\infty$, long inter-arrival times $\tau_i$ of length comparable with $t$ may occur with a probability which is not super-exponentially small in $t$.
Thus $\I(\mu)$ is finite at $\mu=\delta_{(+\infty,+\infty)}$, while the DV functional $\I_0$ is finite only on probability measures supported by $]0,+\infty[^2$.

However $\I_0$ is in general {\it not} a good rate functional on ${\mc P}(]0,+\infty[^2)$ by proposition \ref{p:igood2} and it is good if and only if all exponential moments of $\tau_1$ are finite, i.e. $\xi=+\infty$. As long as one exponential moment of $\tau_1$ is infinite, then the sublevels of $\I_0$ in ${\mc P}(]0,+\infty[^2)$ are not compact, 
and the law ${\bf P}_t$ of $\mu_t$ as $t\to+\infty$ does not satisfy a full large deviations principle on ${\mc P}(]0,+\infty[^2)$. 
%Notice however that $C_t/t$ does satisfy a large deviations principle in any case for a bounded continuous $F$, regardless of the integrability properties of $\tau_1$, and the rate functional $J_F$ can be computed both from $\I$ and from $\I_0$, recall \eqref{I_0J_F} and \eqref{IJ_F}.

There are various extensions of DV theory, dealing with the lack of regularity properties of the Markov process, e.g.\ \cite{Krylov}, or ergodicity 
\cite{Wu1,Wu2}. However, even such extensions do not take into account the model studied in this paper, and at the same time do not 
provide the right large deviations rate functional in this case.

We finally remark that this criticality is not a special feature of $(A_t,B_t)$, but also other processes feature singular behavior. In the same 
setting, one may consider for instance the Markov process $\sigma_t:=(\tau_{N_t}, \frac{t-S_{N_t-1}}{\tau_{N_t}})$. If the tail of $\psi$ has an oscillating behavior, then the empirical measure of $(\sigma_t)_t$ does not  
even satisfy a large deviations principle, but it satisfies optimal upper and a lower large deviations bounds with functionals which may 
be different. 
%Or equivalently, the large deviations rate functional truly depends on the 
%subsequence $(t_k)$ of times along which the limit $t\to +\infty$ is performed. 
This issue is not addressed here and will be the subject of a 
forthcoming work.

\subsection{Affine stretches}\label{affine} In this section we detail how the structure of the rate functional $\I$ explains the appearance of flat stretches in large deviations rate functionals $J_F$.  Let us consider the case of $F\equiv 1$, i.e.\ the large deviations of 
$N_t/t$ as $t\uparrow+\infty$, where $N_t$ is the counting process. Recall that the rate functional is $J_1(m)=m\Lambda^*(1/m)$, where $\Lambda^*(a):=\sup_x(ax-\log \psi(e^{x \tau}))$. Here we suppose that $\xi<+\infty$, i.e.\ that $\psi$ has some infinite exponential moment, and that 
\begin{equation*}
T:=\sup_{c<\xi}\frac{\E(\tau_1 e^{c\tau_1})}{\E(e^{c\tau_1})}<+\infty.
\end{equation*}
%Notice that the function $c\mapsto \frac{\E(\tau_1e^{c\tau_1})}{\E(e^{c\tau_1})}$ is monotone non-decreasing by the convexity of $c\mapsto \log\E(e^{c\tau_1})$. In particular $T\geq \E(\tau_1)=\psi(\tau)$. 
It is then easily seen that $J_1(\cdot)$ is strictly convex on 
$[1/T,+\infty[$, while
\begin{equation*}
J_1(m)=m\Lambda^*(T)+(1-mT)\xi, \qquad m\in[0,1/T].
\end{equation*}
If $\xi=0$ and $T<+\infty$ (which is the case if for instance $\psi$ has polynomial tails and finite mean) , $J_1$ vanishes on $[0,1/T]$.

Therefore, there is a transition between a strictly-convex regime and an affine regime. However, if we go back to the formula \eqref{I_0J_F} above, which becomes for $F\equiv 1$
\begin{equation}
\label{mile}
\begin{split}
 J_1(m) = & \inf\left\{\pi(1/\tau) \H\big(\tilde\pi\, \big|\, \psi\big)\, : \: \pi\in\cP(]0,+\infty[), \, \pi(1/\tau)=m \right\}
 \\ = & m \inf\left\{ \H\big(\zeta\, \big|\, \psi\big)\, : \: \zeta\in\cP(]0,+\infty[), \, \zeta(\tau)=1/m \right\},
 \end{split}
\end{equation}
then it is hard to understand what makes this $\inf$ strictly convex for $m>1/T$ and affine for $m\leq 1/T$. This apparent paradox is solved 
if we take into account formula \eqref{IJ_F} above, which becomes in this case
\begin{equation}
\label{davi}
\begin{split}
 J_1(m) = & \inf\Big\{m \H\big(\zeta\, \big|\, \psi\big) +(1-\alpha)\xi\, : \\ & \qquad \qquad\, \zeta\in\cP(]0,+\infty[), \, \alpha\in[0,1], \, \zeta(\tau)=
\alpha/m \Big\},
 \end{split}
\end{equation}
In \eqref{davi} the appearance of the two regimes is clear.

\begin{itemize}
\item For $m\geq 1/T$, there exists a measure $\zeta_m\in\cP(]0,+\infty[)$ which minimizes the relative entropy $\H\big(\zeta\, \big|\, \psi\big)$ under the constraint $\zeta(\tau)=1/m$, and this minimizer is an exponential tilt of $\psi$, i.e.\
\begin{equation*}
\zeta_m(d\tau)=\frac1{\psi(e^{c(m)\tau})} \, e^{c(m)\tau}\, \psi(d\tau),
\quad \textrm{where $c(m)$ is fixed by} \quad \frac{\psi(\tau\, e^{c(m)\tau})}{\psi(e^{c(m)\tau})}=\frac1m
\end{equation*}
and $\H\big(\zeta_m\, \big|\, \psi\big)=\Lambda^*(1/m)$.
Then the minimizer of \eqref{mile} is $\zeta_m$ and the minimizer of \eqref{davi} is $\zeta_m$ and $\alpha=1$.

\item For $m<1/T$, on the other hand, no minimizer of \eqref{mile} exists and the additional parameter $\alpha$ in \eqref{davi} starts to play a 
role; it turns out that the minimizer of \eqref{davi} is given by $\alpha_m=Tm$ and $\zeta_{1/T}$, and therefore we obtain the correct value of 
$J_1(m)$. 
\end{itemize}
The same picture is correct for more general functions $F$. Although $J_F$ can be expressed as an $\inf$ in terms of $\I_0$, in general 
this $\inf$ is not attained and it is not easy to guess a minimizing sequence; on the other hand this problem is easily solved if one expresses 
$J_F$ as a $\min$ in terms of $\I$ over a larger set of probability measures. 

This phenomenon is discussed in detail in \cite{LMZ1} with applications to a heat conduction model. Although the results of this paper are not explicitly applied there, the intuition behind the proof of \cite[Theorem 3.4]{LMZ1} comes from the understanding of the structure of the functional $\I$ defined above.

For more on minimization of entropy functionals, see \cite{csiszar}.

\section{The functional $\I$}
In this section we analyze the properties of the functional $\I$ and prove in particular Proposition~\ref{p:igood2}. We also prove the following 
stability result which will come useful in the following. Recall the definitions  \eqref{e:muomega0}, \eqref{e:muomega}, \eqref{e:xiphi} and 
\eqref{e:I2}. Then
\begin{prop}
\label{p:igood}
Let $(\psi_n)$ be a sequence in $\mc P(]0,+\infty[)$. Let $\xi_n$ and $\I_n$ be defined as in \eqref{e:xiphi} and \eqref{e:I2} respectively, with $
\psi$ replaced by $\psi_n$. Assume that $\psi_n\rightharpoonup\psi$ and $\xi_n \to \xi$ as $n\to +\infty$. Then
\begin{itemize}
\item[(1)] Any sequence $(\mu_n)$ in $\mc P(]0,+\infty]^2)$ such that $\varlimsup_n \I_n(\mu_n) <+\infty$ is tight, and thus precompact  in $
\mc P(]0,+\infty]^2)$.

\item[(2)] For any $\mu$ and any sequence $(\mu_n)$ in $\mc P(]0,+\infty]^2)$ such that $\mu_n \rightharpoonup \mu$, we have $\varliminf_n 
\I_n(\mu_n) \ge \I(\mu)$.

\item[(3)] For any $\mu$  in $\mc P(]0,+\infty]^2)$ with $\I(\mu)<+\infty$, there exists a sequence $(\mu_n)$ such that $\mu_n  \rightharpoonup 
\mu$, $\mu_n \in \Delta_0$ for all $n$, and $\varlimsup_n \I_n(\mu_n) \le \I(\mu)$.

\end{itemize}
\end{prop}
In the setting of \cite{DalMaso}, Proposition~\ref{p:igood} states that $\I_n$ $\Gamma$-converges to $\I$, and that $\Delta_0$ is $\I$-dense in 
$\mc P(]0,+\infty]^2)$. 
Before proving Proposition~\ref{p:igood}, let us show how Proposition~\ref{p:igood2} follows immediately from it.
\begin{proof}[Proof of Proposition~\ref{p:igood2}]
In Proposition~\ref{p:igood} take $\psi_n=\psi$. Then the statement (1) implies that $\I$ has precompact sublevel set (namely it is coercive), 
statement (2) implies that $\I$ has closed sublevel sets (namely it is lower semicontinuous), and thus (1) and (2) imply that $\I$ is good. Since 
$\I \le \I_0$, statement (2) implies that $\I$ is smaller or equal than the lower semicontinuous envelope of $\I_0$, while (3) states that $\I$ is 
greater or equal to it.
\end{proof}

\begin{lemma}
\label{algfact} For all $\pi\in \cP(]0,+\infty[)$ such that
$\pi(1/\tau)<+\infty$ and $a>0$
\begin{equation*}
\begin{split}
\pi(1/\tau)\, \H\big(\tilde\pi\, \big|\, \psi\big) & =
\sup_{\varphi} \left(\pi(\varphi/\tau) -\pi(1/\tau)\log\psi (e^\varphi)\right)
\\ & = \sup_{\varphi\,:\:
\psi (e^\varphi)=a} \left(\pi(\varphi/\tau) -\pi(1/\tau)\log\psi (e^\varphi)\right)
\\ & 
= \sup_f  \left(\pi(f) -\pi(1/\tau)\log\psi (e^{\tau f})\right)
\end{split}
\end{equation*}
where the suprema are taken over $\varphi\in C_b(]0,+\infty[)$, and over and $f \in C(]0,+\infty[)$ bounded from below and such that $\pi(f)<+
\infty$.

In particular $\pi \mapsto \pi(\tau) \H(\pi|\psi)$ is convex on $\{\pi \in \mc P(]0,+\infty[)\,:\: \pi(1/\tau)<+\infty\}$ and thus $\I$ is convex.
\end{lemma}
\begin{proof} It is well known that
\begin{equation*}
\H\big(\tilde\pi\, \big|\, \psi\big) = \sup_{\varphi\in C_b(]0,+\infty[)}
\left(\tilde\pi(\varphi) -\log\psi (e^\varphi)\right).
\end{equation*}
Now, suppose that $\psi(e^\varphi)=a>0$ and set $\varphi_a:=\varphi-\log a$. Then
\begin{equation*}
\pi(\varphi/\tau) -\pi(1/\tau)\log\psi (e^\varphi) = \pi(\varphi_a/\tau) -
\pi(1/\tau)\log\psi (e^{\varphi_a})
\end{equation*}
and $\psi(e^{\varphi_a})=1$. Therefore the quantity
\begin{equation*}
\sup_{\varphi\,:\:
\psi (e^\varphi)=a} \left(\pi(\varphi/\tau) -\pi(1/\tau)\log\psi (e^\varphi)\right)
= \sup_{\varphi\,:\:
\psi (e^\varphi)=1} \left(\pi(\varphi/\tau) -\pi(1/\tau)\log\psi (e^\varphi)\right)
\end{equation*}
does not depend on $a>0$ and thus
\begin{equation*}
\begin{split}
  \sup_a \sup_{\varphi\,:\:
\psi (e^\varphi)=a} \left(\pi(\varphi/\tau) -\pi(1/\tau)\log\psi (e^\varphi)\right)
& = \sup_{\varphi} \left(\pi(\varphi/\tau) -\pi(1/\tau)\log\psi (e^\varphi)\right)
\\ & =
\pi(1/\tau)\, \H\big(\tilde\pi\, \big|\, \psi\big),
\end{split}
\end{equation*}
where all suprema are taken over $\varphi\in C_b(]0,+\infty[)$.

A standard approximation argument proves that one can take $\varphi = \tau f$ in the supremum, provided the conditions on $f$ in the 
statement of the lemma hold.
\end{proof}

\begin{proof}[Proof of Proposition~\ref{p:igood}-(1)]
Since $\varlimsup_n \I_n(\mu_n)<+\infty$, $\mu_n \in \Delta$ for $n$ large enough, and thus $\mu_n$ admits the writing \eqref{e:muomega0}-
\eqref{e:muomega}, for some $\alpha_n \in [0,1]$ and $\pi_n \in \mc P(]0,+\infty[)$ with $\pi_n(1/\tau)<+\infty$.
We first show that 
\begin{equation}
\label{e:pibound}
\varlimsup_n \alpha_n \pi_n(1/\tau)<+\infty.
\end{equation}
Notice that
\[
\I_n(\mu_n) \ge \alpha_n \pi_n(1/\tau) \H(\tilde \pi_n|\psi_n),
\qquad
\pi_n(1/\tau) = \frac{1}{\tilde{\pi}_n(\tau)},
\]
so that
\begin{equation*}
\begin{split}
\varlimsup_n \alpha_n \pi_n(1/\tau) \le \varlimsup_n  \frac{\I_n(\mu_n)}{\H(\pi_n|\psi_n)} \wedge \frac{\alpha_n}{\tilde{\pi}_n(\tau)} \le \frac{C}
{\varliminf_n \H(\tilde \pi_n|\psi_n) \vee \tilde \pi_n(\tau)}
\end{split}
\end{equation*}
for some $C>0$. The denominator in the right hand side above is uniformly bounded away from $0$. Indeed, if  $\lim_k \H(\tilde \pi_{n_k}|
\psi_{n_k})$ vanishes on some subsequence $n_k$, then $\lim_k \tilde \pi_{n_k} = \lim_k \psi_{n_k}=\psi$, and therefore $\varliminf_k \tilde 
\pi_{n_k}(\tau) \ge \psi (\tau) >0$ and \eqref{e:pibound} holds. Thus $(\mu_n)$ is precompact.

It is easy to see that for each $M>0$ the set $\Delta_M:= \{\mu \in \Delta \,:\: \mu(1/(a+b) \le M\}$ is compact in $\mc P(]0,+\infty]^2)$. Now by 
\eqref{e:pibound}
\begin{equation*}
\varlimsup_n \mu_n(1/(a+b)) = \varlimsup_n \alpha_n \pi_n(1/\tau) <+\infty
\end{equation*}
namely $\mu_n \in \Delta_M$ for $n$ and $M$ large enough.
\end{proof}

\begin{lemma}
\label{l:realineq}
Let $\pi_n \in \mc P(]0,+\infty[)$ be such that $\pi_n(1/\tau)<+\infty$ and
\begin{equation}
\label{e:pinconv}
\lim_n \pi_n = \beta \pi + (1-\beta)\delta_{+\infty}
\end{equation}
for some $\beta \in [0,1]$ and $\pi  \in \mc P(]0,+\infty[)$ such that $\pi(1/\tau)<+\infty$ .
Then
\begin{equation}
\label{e:realineq}
\varliminf_n \pi_n(1/\tau) \H(\tilde \pi_n|\psi_n) \ge \beta \pi(1/\tau) \H(\tilde \pi|\psi) + (1-\beta)\xi .
\end{equation}
\end{lemma}
\begin{proof}
By Lemma~\ref{algfact}
\begin{equation}
\label{e:algfact}
\begin{split}
\pi_n(1/\tau) \H(\tilde \pi_n|\psi_n) 
 & = \sup_{f} \left(\pi_n(f) -\pi_n(1/\tau)\log\psi_n (e^{\tau f})\right)
\end{split}
\end{equation}
where the supremum is carried over continuous functions $f$ bounded from below and such that $\pi_n(f)<+\infty$.

Fix a $\varphi \in C_b( ]0,+\infty[)$ such that $\psi(e^\varphi)<1$. Fix also $c\in [0,\xi[$ if $\xi>0$ or take $c=0$ if $\xi=0$. For an arbitrary 
$M>0$, let $\chi_M$ be a smooth function on $]0,+\infty[$ such that
\begin{eqnarray*}
& & \chi_M(\tau)=1 \qquad \text{for $\tau \le 1/(M+1)$ or $\tau \ge M+1$},
\\
& & \chi_M(\tau)=0\qquad \text{for $1/M \le \tau \le M$}.
\end{eqnarray*}
Since $\psi(e^\varphi)<1$, there exists $M'\equiv M'(\varphi,c)$ such that 
\begin{equation*}
\psi(e^{c \,\tau \,\chi_M + \varphi\, (1-\chi_M) })<1 \qquad \qquad \text{$M\ge M'(\varphi,c)$}
\end{equation*}
and since $\psi_n \to \psi$ and $\xi_n \to \xi>c$ (the case $\xi=0$ is easily taken care), for $n$ large enough depending on $M$, $\varphi$ 
and $c$
\begin{equation}
\label{e:Mlarge}
\psi_n(e^{c \,\tau \,\chi_M + \varphi\, (1-\chi_M) })<1 \qquad \qquad \text{$M\ge M'(\varphi,c)$ and $n$ large enough.}
\end{equation}
Now, in \eqref{e:algfact} consider a $f$ of the form $f(\tau)=c \,\chi_M(\tau) + \varphi(\tau)\, (1-\chi_M(\tau))/\tau$, which is allowed for $n$ large 
enough such that \eqref{e:Mlarge} holds. Then the logarithm in the right hand side of \eqref{e:algfact} is negative, and therefore recalling 
\eqref{e:pinconv}
\begin{equation*}
\begin{split}
\varliminf_n \pi_n(1/\tau) \H(\tilde \pi_n|\psi_n) 
 & \ge \varliminf_n \pi_n(\varphi (1-\chi_M)/\tau) + \pi_n(c \chi_M)
\\
& = \beta \pi(\varphi (1-\chi_M)/\tau ) + \beta\,c\,\pi(\chi_M)+  (1-\beta)\, c .
\end{split}
\end{equation*}
Taking the limit $M\to \infty$, since $\pi(1/\tau)<+\infty$ and $\pi(\{+\infty\})=0$, by dominated convergence
\begin{equation*}
\varliminf_n \pi_n(1/\tau) \H(\tilde \pi_n|\psi_n) 
 \ge  \beta \pi(\varphi/\tau) + (1-\beta)\,c.
\end{equation*}
Optimizing over $c<\xi$ and $\varphi$ such that $\psi(e^\varphi)<1$
\begin{equation}
\label{e:pinpin}
\begin{split}
\varliminf_n \pi_n(1/\tau) \H(\tilde \pi_n|\psi_n) 
 & \ge \sup_{c < \xi} \sup_{\varphi}  \beta \pi(\varphi/\tau) + (1-\beta)\,c
\\
& = \beta \sup_{\varphi} \pi(\varphi/\tau) + (1-\beta)\xi.
\end{split}
\end{equation}
Still by Lemma \eqref{algfact}
\begin{equation*}
\begin{split}
\sup_{\psi(e^\varphi)<1} \pi(\varphi/\tau)  
& =\sup_{
a<1} \sup_{\psi(e^\varphi)=a } \pi(\varphi/\tau) 
\\ & 
= \sup_{a<1} \big[
\log a + \sup_{\psi(e^\varphi)=a} \pi(\varphi/\tau) - \log \psi(e^\varphi) \big]
\\ & 
= \sup_{a<1}
\log a +  \pi(1/\tau) \H(\tilde \pi|\psi)= \pi(1/\tau) \H(\tilde \pi|\psi)
\end{split}
\end{equation*}
which concludes the proof in view of \eqref{e:pinpin}.
\end{proof}

\begin{proof}[Proof of Proposition~\ref{p:igood}-(2)]
First note that it is enough to prove the statement for a subsequence of $(\mu_n)$, and subsequences will be often indexed by the same $n$ 
in this proof. Therefore one can assume $\sup_n \I_n(\mu_n) <+\infty$, the statement being trivial otherwise. Thus, up to passing to a 
subsequence, $\mu_n \in \Delta$ and according to \eqref{e:muomega0}-\eqref{e:muomega} one can write
\begin{equation}
\label{e:munmu0}
\begin{split}
& \mu_n= \alpha_n 
\mu_{0,n} + (1-\alpha_n)\, \delta_{(+\infty,+\infty)},
\\
& \mu_{0,n}:=\int_{[0,1]\times]0,+\infty[} \delta_{(u\tau,(1-u)\tau)}\, du \otimes \pi_n(d\tau),
\end{split}
\end{equation}
for some $\alpha_n \in [0,1]$ and $\pi_n \in \mc P(]0,+\infty[)$ with $\pi_n(1/\tau)<+\infty$. If $\varlimsup_n \alpha_n = 0$, then $\mu=\lim_n 
\mu_n= \delta_{(+\infty,+\infty)}$ and therefore
\begin{equation*}
\varliminf_n \I_n(\mu_n) \ge \varliminf_n (1-\alpha_n)\xi_n = \xi = \I(\mu).
\end{equation*}
Let us turn to the case $\varlimsup_n \alpha_n =: \bar \alpha >0$. Up to passing to a subsequence, one can assume $\lim_n \alpha_n=\bar 
\alpha>0$. Since $\sup_n \I_n(\mu_n) <+\infty$, the bound on \eqref{e:pibound} holds, and since $\bar \alpha>0$ it yields
\begin{equation*}
\varlimsup_n \pi_n(1/\tau)<+\infty.
\end{equation*}
In particular $\pi_n$ is tight in $\mc P(]0,+\infty])$ (note that $+\infty$ is and should be included here). Thus, up to passing to a further 
subsequence
\begin{equation}
\label{e:pinconvbeta}
\lim_n \pi_n = \beta \pi + (1-\beta)\delta_{+\infty}
\end{equation}
for some $\beta \in [0,1]$. If $\beta>0$ by \eqref{e:pibound}
\begin{equation*}
\label{e:pinconv2}
\pi(1/\tau) \le  \frac{1}{\beta} \varlimsup_n \pi_n(1/\tau)<+\infty
\end{equation*}
while one can choose an arbitrary $\pi$ satisfying $\pi(1/\tau)<+\infty$ if $\beta=0$. In particular the conditions of Lemma~\ref{l:realineq} are 
fulfilled, and therefore \eqref{e:realineq} holds.

Patching \eqref{e:munmu0} and \eqref{e:pibound} together
\begin{equation*}
\begin{split}
& \mu = \lim_n \mu_n = \bar \alpha \beta \mu_0 + (1-\bar \alpha \beta) \delta_{(+\infty,+\infty)},
\\
& \mu_{0}:=\int_{[0,1]\times]0,+\infty[} \delta_{(u\tau,(1-u)\tau)}\, du \otimes \pi(d\tau).
\end{split}
\end{equation*}
In particular $\mu \in \Delta$ with $\alpha=\bar \alpha \beta$. And recalling $\alpha_n \to \bar \alpha$ and $\xi_n \to \xi$
\begin{equation*}
\begin{split}
\I(\mu) 
& =\bar \alpha \beta \pi(1/\tau) \H(\tilde \pi|\psi) + (1-\bar \alpha \beta) \xi
\\ & 
=
\bar \alpha \big[\beta \pi(1/\tau) \H(\tilde \pi|\psi) + (1-\beta)\xi \big]
+   (1-\bar \alpha )\xi 
\\ & 
= \bar \alpha \pi_n(1/\tau) \H(\tilde \pi_n|\psi_n)  + (1-\bar \alpha) \xi
\\ & \phantom{=}
+ \bar \alpha \big [\beta \pi(1/\tau) \H(\tilde \pi|\psi) + (1-\beta)\xi 
- \pi_n(1/\tau) \H(\tilde \pi_n|\psi_n)  \big]
\\ & \le \varliminf_n \I_n(\mu_n)
  + \bar \alpha\varliminf_n  \big [\beta \pi(1/\tau) \H(\tilde \pi|\psi) + (1-\beta)\xi 
- \pi_n(1/\tau) \H(\tilde \pi_n|\psi_n)  \big].
\end{split}
\end{equation*}
The limit in square brackets in the last line is negative, by \eqref{e:pinconvbeta} and Lemma~\ref{l:realineq}. The wanted inequality follows.
\end{proof}

\begin{proof}[Proof of Proposition~\ref{p:igood}-(3)] Since $\I(\mu)<+\infty$, $\mu \in \Delta$ and let $\alpha$ and $\pi$ be as in 
\eqref{e:muomega0}-\eqref{e:muomega} (again, the choice of $\pi$ is not relevant if $\alpha =0$).

Fix $\delta >0,\,L>M>1$ such that $\psi(\{1/M\}) = \psi(\{M\})=0 = \psi(\{L\})=0$. Then there exist $N \in \mathbb N$ and $1/M = T_1 < T_2<\ldots 
<T_N = M$ such that $T_{i+1}-T_i \le \delta$ and $\psi(\{T_i\})=0$ for all $i=1,\ldots,N$. Here of course $N \equiv N(M,\delta)$ and $T_i \equiv 
T_i(M,\delta)$; we also use the shorthand notation $A_i=[T_i,T_{i+1}[$ and $A = \cup_{i=1}^N A_i$ in this proof. Then for $L>M$ define $
\pi_n^{\delta,M,L}(d\tau) \in \mc P(]0,+\infty[)$ as
\begin{equation*}
\pi_n^{\delta,M,L}(d\tau)= \frac{\tau \tilde \pi_n^{\delta,M,L}(d\tau)}{\tilde \pi_n^{\delta,M,L}(\tau)},
\end{equation*}
\begin{equation*}
\tilde \pi_n^{\delta,M,L}(d\tau) =\alpha \sum_{i=1}^N \frac{\tilde \pi(A_i)}{\tilde \pi(A)} \psi_n (d\tau | A_i)
+ (1-\alpha) \psi_n(d\tau | [M,L[).
\end{equation*}
The above definition is well posed if $L>M$ is large enough, and $n$ is large enough depending on $L$ and $M$ ($n$ will be sent to $+\infty
$ before $L$, and $L$ before $M$). Indeed, since $\I(\mu)<+\infty$ and $\psi(\partial A_i)=0$, if $\psi_n(A_i)=0$ for $n$ large, then $\pi 
(A_i)=0$, and similarly if $\psi_n([M,L[)=0$ then $\alpha=1$.

We want to prove
\begin{equation}
\label{e:pin1}
\lim_{M\to +\infty} \lim_{L\to +\infty} \lim_{\delta \downarrow 0} \lim_n \pi_n^{\delta,M,L}(d\tau) =  \alpha \pi + (1-\alpha) \delta_{+\infty},
\end{equation}
\begin{equation}
\label{e:pin2}
\varliminf_M \varliminf_L \varliminf_\delta \varlimsup_n \pi_n^{\delta,M,L}(1/\tau) \H( \tilde \pi_n^{\delta,M,L} | \psi_n)
 \le \alpha \pi(1/\tau) \H( \tilde \pi | \psi)+ (1-\alpha) \xi = \I(\mu),
\end{equation}
where the limits in $M$ and $L$ are understood to run over $M$ and $L$ satisfying the above conditions.

Indeed, once \eqref{e:pin1}-\eqref{e:pin2} are proved, one can extract subsequences $\delta_n \to 0$, $L_n,\,M_n \to +\infty$ such that,
 defining $\pi_n:= \pi_n^{\delta_n,M_n,L_n}$, one has $\pi_n \rightharpoonup \pi$ and also $\varlimsup_n
 \pi_n(1/\tau) \H(\tilde \pi_n| \psi_n) \le \I(\mu)$. It is then easy to verify that the sequence $(\mu_n)$ defined by \begin{equation*}
\mu_n:=\int_{[0,1]\times]0,+\infty[} \delta_{(u\tau,(1-u)\tau)}(da,db)\, du \otimes \pi_n(d\tau)
\end{equation*}
fullfills the wanted requirements.

Note that the convergence $\tilde \pi_n^{\delta,M,L}  \rightharpoonup \tilde \pi$ is immediate, so that \eqref{e:pin1} readily follows. In order to 
prove \eqref{e:pin2} define
\begin{equation*}
\begin{split}
& \tilde \pi_{0,n}(d\tau) \equiv \tilde \pi_{0,n}^{\delta,M} (d\tau)= \sum_{i=1}^N \frac{\tilde \pi(A_i)}{\tilde \pi(A)}
  \psi_n (d\tau | A_i),
\\ &
\tilde \pi^{\delta,M}_{0}(d\tau) \equiv \tilde \pi^{\delta,M}_{0}(d\tau) = \sum_{i=1}^N \frac{\tilde \pi(A_i)}{\tilde \pi(A)}
  \psi (d\tau | A_i).
\end{split}
\end{equation*}
By the convexity statement in Lemma~\ref{algfact}
\begin{equation}
\label{e:ient1}
\begin{split}
& \pi_n^{\delta,M,L}(1/\tau) \H(\tilde \pi_n^{\delta,M,L}| \psi_n) 
 = \frac{1}{\tilde \pi_n^{\delta,M,L}(\tau)} \H(\tilde \pi_n^{\delta,M,L}| \psi_n)  
\\ & \qquad  
\le \alpha \frac{1}{
\tilde \pi_{0,n}(\tau)}   \H(\tilde \pi_{0,n}| \psi_n)+ (1-\alpha) \frac{1}{\psi_n(\tau|[M,L[)} \H( \psi_n(\cdot|[M,L[)|\psi_n).
\end{split}
\end{equation}
All the terms above can be explicitly calculated. In particular, since $\psi(\{M\})=\psi(\{L\})=0$,
\begin{equation}
\label{e:ient2}
\lim_n \frac{1}{\psi_n(\tau|[M,L[)} \H( \psi_n(\cdot|[M,L[)|\psi_n) = \frac{1}{\psi(\tau|[M,L[)} \H( \psi(\cdot|[M,L[)|\psi),
\end{equation}  
and it is easy to check that
\begin{equation}
\label{e:ient3}
\begin{split}
& \varliminf_M \varliminf_L
  \frac{1}{\psi(\tau|[M,L[)} \H( \psi(\cdot|[M,L[)|\psi)
%\\ & \qquad \qquad
 \le  - \varlimsup_M \frac{1}{M} \log \psi([M,+\infty[) = \xi.
\end{split}
\end{equation}
On the other hand, since $\psi(\partial A_i)=0$, one has
\begin{equation*}
\begin{split}
& \lim_n \pi_{0,n}(\tau) \to \tilde \pi_0(\tau),
\qquad \lim_{\delta \downarrow 0} \tilde \pi_0(\tau)= \tilde \pi(\tau|[1/M,M[),
\\
& \lim_{M \to +\infty}\tilde \pi(\tau | [1/M,M[)= \tilde \pi(\tau),
\end{split}
\end{equation*}
namely
\begin{equation}
\label{e:ient4}
\lim_{M\to +\infty} \lim_{\delta \downarrow 0} \lim_n \tilde \pi_{0,n}(\tau)=\tilde \pi(\tau),
\end{equation}
and
\begin{equation*}
\begin{split}
 \lim_n  \H(\tilde \pi_{0,n}| \psi_n) & = \H(\tilde \pi_0 | \psi) = \sum_{i=1}^N
\frac{\tilde \pi(A_i)}{\tilde \pi(A)} \log  \frac{\tilde \pi(A_i)}{\tilde \pi(A) \psi(A_i)}
 \\ &
=  \frac{1}{\tilde \pi(A)} 
\Big[\tilde \pi(A^c) \log \frac{\tilde \pi(A^c)}{ \psi(A^c)}+ \sum_{i=1}^N \tilde \pi(A_i) \log  \frac{\tilde \pi(A_i)}{ \psi(A_i)}\Big]
\\ & \phantom{=}
- \big[ \log  \tilde \pi(A) + \tilde \pi(A^c) \log \frac{\tilde \pi(A^c)}{ \psi(A^c)}]
\\ &
\le \H(\tilde \pi|\psi) - \big[ \log  \tilde \pi(A) + \tilde \pi(A^c) \log \pi(A^c)].
\end{split} 
\end{equation*}
The term in square brackets in the last line above vanishes as $M \to +\infty$, so that
\begin{equation}
\label{e:ient5}
\varlimsup_M \sup_{\delta<1} \lim_n  \H(\tilde \pi_{0,n}| \psi_n) \le  \H(\tilde \pi \,|\, \psi). 
\end{equation}
The inequality \eqref{e:pin2} finally follows from \eqref{e:ient1}, \eqref{e:ient2}, \eqref{e:ient3}, \eqref{e:ient4}, \eqref{e:ient5}.
\end{proof}
This concludes the proof of Proposition~\ref{p:igood}. We end this section with some additional results concerning the functional $\I$ which 
will come useful in the following.

\begin{lemma}\label{closed}
The set $\Delta$ defined in \eqref{e:muomega} is closed
in $\cP(]0,+\infty]^2)$.
\end{lemma}
\begin{proof}
Let $\mu_n\in\Delta$ such that
$\mu_n\rightharpoonup\mu\in\cP(]0,+\infty]^2)$ in $]0,+\infty]^2$, with $\mu_n$ given by \eqref{e:muomega0}-\eqref{e:muomega} with $
\alpha_n\in[0,1]$ and $\pi_n\in\cP(]0,+\infty[)$.
We can assume that $\alpha_n$ converges to some $\bar \alpha$ and $\pi_n\rightharpoonup\pi\in\cP([0,+\infty])$. By Skorohod's 
representation theorem, there exists a sequence $(P_n)_n$ of random variables such that $P_n$ has law $\pi_n$, $P_n\in\, ]0,+\infty[$ 
converges a.s. to $P\in[0,+\infty]$ and $P$ has law $\pi$. If $U$ is uniform on $[0,1]$ and independent of $(P_n)_n$ then
for any $f\in C_b([0,+\infty]^2)$ we obtain that
\begin{equation*}
\begin{split}
\mu_n(f) & =\alpha_n\E(f(UP_n,(1-U)P_n))+(1-\alpha_n)f(+\infty,+\infty)
\\ & \to
\alpha\E(f(UP,(1-U)P))+(1-\alpha)\, f(+\infty,+\infty)
\end{split}
\end{equation*}
and this limit must be equal to $\mu(f)$. Since $\mu\in\cP(]0,+\infty]^2)$, then $\bbP(P=0)=0$. If $\bbP(P<+\infty)\in\{0,1\}$ then $\mu\in\Delta
$. If $\beta:=\bbP(P<+\infty)\in\,]0,1[$ then
\begin{equation*}
\mu(f) = \alpha\beta \, \E(f(UP,(1-U)P)\, | \, P<+\infty) +
(1-\alpha\beta)\, f(+\infty,+\infty)
\end{equation*}
and therefore $\mu\in\Delta$.
\end{proof}

For a bounded measurable $f\colon ]0,+\infty[\,\times\,]0,+\infty[ \to\R$ set
\begin{equation}\label{overf}
\overline f(r,\tau):=\int_0^r f(u\tau,(1-u)\tau)\, du, \qquad
r\in[0,1], \, \tau>0.
\end{equation}

Let $\Gamma$ be the set of all bounded lower semicontinuous $f\colon  ]0,+\infty]\,\times\,]0,+\infty] \to\R$ such that
\begin{equation}
\label{calpha}
C_f:=
\int_{]0,+\infty[} \psi(d\tau)\, e^{\tau\overline f(1,\tau)} < 1.
\end{equation}
\begin{equation}
\label{e:fsmall}
D_f:=\sup_{s>0} \int_{]s,+\infty[} \psi(d\tau)\,e^{\tau \overline f(s/\tau,\tau) }   <+\infty.
\end{equation}

\begin{lemma}
\label{l:3.4}
For all $\mu \in \Delta$
\begin{equation} \label{e:muf}
  \I(\mu) \le \sup_{f\in\Gamma} \mu(f)
\end{equation}
\end{lemma}

\begin{proof}
  Let $\varphi \in C_\mathrm{c}(]0,+\infty])$, $c<\xi$ if $\xi>0$ and $c:=0$ if $\xi=0$ and
  $M>0$. Let
\begin{equation*}
   f_{c,\varphi,M}(a,b):=   \frac{\varphi(a+b)}{a+b} + c\,\un {]M,+\infty]}(a+b),
   \qquad (a,b)\in ]0,+\infty]^2.
  \end{equation*}
Then $f_{c,\varphi,M}$ is lower semicontinuous on $]0,+\infty]^2$ and
\begin{equation*}
\overline  f_{c,\varphi,M}(r,\tau):= r\left( \frac{\varphi(\tau)}\tau + c\,\un {]M,+\infty]}(\tau)\right), \qquad \tau>0, \, r\in[0,1].
\end{equation*}
Then
\begin{equation*}
  \begin{split}
 \int_{[s,+\infty[} \psi(d\tau)\,e^{\tau\overline f_{c,\varphi,M}(s/\tau,\tau)} 
& = \int_{[s,+\infty[} \psi(d\tau)\,\exp\left(\frac s\tau\left(\varphi(\tau) + c \tau\un {]M,+\infty]}(\tau)\right)\right)
\\ &  \le  e^{\|\varphi\|_{\infty}}  \psi(e^{c\tau})
 \end{split}
\end{equation*}
which is bounded uniformly in $s$, so that \eqref{e:fsmall} holds for
$f=f_{c,\varphi,M}$. Let now $a<1$. If
\begin{equation}\label{e:varphismall}
\psi(e^\varphi) =a<1
\end{equation}
then there exists $M_0=M_0(c,\varphi)$ such that for all $M>M_0$
\begin{equation*}
C_{f_{c,\varphi,M}} = \psi\left( e^{\varphi + c\tau 1_{]M,+\infty]}}\right)<1
\end{equation*}
and therefore $f_{c,\varphi,M}\in\Gamma$.
Now, if $\mu$ is given by \eqref{e:muomega0}-\eqref{e:muomega} then
\begin{equation*}
\mu(f_{c,\varphi,M}) =
 {\alpha}\,\pi(\varphi/\tau) + c\,(1-\alpha).
 \end{equation*}
 Since $\pi(\varphi/\tau)=\pi(1/\tau) \,\tilde\pi(\varphi)$ then (with the usual convention $0\cdot\infty=0$)
 \begin{equation*}
 \begin{split}
      \sup_{f\in\Gamma} \mu(f) &  \geq
    \sup_{\varphi} \sup_{c,m}\sup_{M} \mu(f_{c,\varphi,M})
    \\
 & = \alpha\,\pi(1/\tau)\, \sup_{\varphi} \left\{\tilde\pi(\varphi)-\log\psi(e^\varphi)\right\} + \pi(1/\tau)\, \log a+(1-\alpha)\, \xi
\end{split}
\end{equation*}
where in the right hand side, the supremum on $M$ is performed
over $]M_0(c,\varphi),+\infty[$, the supremum on $c$ over $[0,\xi[$ and the supremum on $\varphi$ over $\varphi \in C_\mathrm{c}(]0,+\infty])$
satisfying \eqref{e:varphismall}. By Lemma~\ref{algfact} the supremum over $\varphi$ satisfying
\eqref{e:varphismall} does not depend on $a$ and the first term equals $\alpha\, \pi(1/\tau)\H(\tilde\pi\,|\,\psi)$, so that
optimizing over $a$
\begin{equation*}
\sup_{f\in\Gamma} \mu(f)   \geq \sup_{a<1} \left\{\alpha\, \pi(1/\tau)\H(\tilde\pi\,|\,\psi)+(1-\alpha)\xi
 + \alpha\,\pi(1/\tau)\, \log a\right\} = \I(\mu).
\end{equation*}
\end{proof}

\section{Upper bound}
In this section we prove the upper bound \eqref{upper} in Theorem 
\ref{t:ld1}.

\subsection{Exponential tightness}
\begin{lemma}
\label{l:3.2}
\begin{equation}
\label{e:onemom}
\lim_{M\to +\infty} \varlimsup_{t\to +\infty}
\frac{1}{t}  \log \bbP (\mu_t(1/(a+b))>M)=-\infty.
\end{equation}
In particular the sequence $({\bf P}_t)_{t>0}$ is exponentially tight
with speed $t$, namely
\begin{equation*}
\inf_{\mc K \subset \subset \mc P(]0,+\infty]^2)} \varlimsup_{t\to+\infty}
\frac{1}{t}  \log {\bf P}_t (\mc K)=-\infty.
\end{equation*}
\end{lemma}
\begin{proof} We recall that $\{S_n\leq t\}=\{N_t> n\}$.
Note that if $\lfloor Mt\rfloor \ge 1$
\begin{equation*}
\begin{split}
\{\mu_t(1/\tau) > M\}
= \left\{\frac{N_t-1}t + \frac{t-S_{N_t}}{t\,\tau_{N_t}} > M \right\}
\subset \big\{N_t > \lfloor Mt\rfloor\big\}=
\left\{S_{\lfloor Mt\rfloor}  \leq t \right\}.
\end{split}
\end{equation*}
Therefore by the Markov inequality
\begin{equation*}
\begin{split}
\bbP(\mu_t(1/\tau) \ge M) \le
 \bbP\left(S_{\lfloor Mt\rfloor} \le t\right)
\le e^t \, \E\left( e^{-S_{\lfloor Mt\rfloor}}\right)
=  e^{t + \lfloor Mt\rfloor \log c}
\end{split}
\end{equation*}
where $c:= \E\left( e^{-\tau_1} \right)<1$, and
inequality \eqref{e:onemom} follows easily. Since for any $M>0$ the set $\{\mu \in
\mc P(]0,+\infty]^2)\,:\: \mu(1/(a+b)) \le M \}$ is tight in $]0,+\infty]^2$,
exponential tightness follows.
\end{proof}

\subsection{The empirical measure is asymptotically close to $\Delta$}
We give here the main argument to show that the rate functional at speed $t$ of $\mu_t$ must be equal to $+\infty$ outside $\Delta$. It will 
follow from the Lemma~\ref{closed}, and the following Lemma stating that $\mu_t$ belongs to an arbitrary neighborhood of $\Delta$ in $
\cP(]0,+\infty]^2)$
for $t$ large enough.
\begin{lemma}\label{nut}
For $f\in C_b(]0,+\infty]^2)$, set
\begin{equation}
\label{nuteq}
\nu_t(f) :=\frac1t \sum_{i=1}^{N_t-1}  {\tau_i} \, \int_0^1 f(u\tau_i,(1-u)\tau_i)\, du\,
+\frac{t-S_{N_t-1}}{t}\, f\left(+\infty,+\infty \right)
\end{equation}
then $\nu_t\in\Delta$. For all $f\in C_b(]0,+\infty]^2)$ and $\delta>0$, there exists $t$ large enough such that the event $\{|\mu_t(f)-\nu_t(f)|>
\delta\}$ is empty.
\end{lemma}
\begin{proof}
It is easy to see that $\nu_t\in\Delta$, and that it is given as in \eqref{e:muomega0}-\eqref{e:muomega} with
\begin{equation*}
\alpha=\frac{S_{N_t-1}}t, \qquad \pi=\frac1{S_{N_t-1}}\sum_{i=1}^{N_t-1}  {\tau_i}\, \delta_{\tau_i}.
\end{equation*}

Recall the definition \eqref{overf}. Then for all $f\in C_b(]0,+\infty]^2)$
\begin{equation}
\label{1}
\begin{split}
\mu_t(f)  =\frac1t \sum_{i=1}^{N_t-1}  {\tau_i} \int_0^1 f(u\tau_i,(1-u)\tau_i)\, du
+ \frac{\tau_{N_t}}{t}\int_0^{\frac{t-S_{N_t-1}}{\tau_{N_t}}} f(u\tau_{N_t},(1-u)\tau_{N_t})\, du
\\ = \frac1t \, \sum_{i=1}^{N_t-1}  {\tau_i} \, \overline f(1,\tau_i)\,
+\frac{\tau_{N_t}}{t}\, \overline f\left(\frac{t-S_{N_t-1}}{\tau_{N_t}},\tau_{N_t} \right).
\end{split}
\end{equation}
We can rewrite
\begin{equation*}
\frac{\tau_{N_t}}{t}\, \overline f\left(\frac{t-S_{N_t-1}}{\tau_{N_t}},\tau_{N_t} \right)=\frac{t-S_{N_t-1}}{t}\int_0^1 f\left(u(t-S_{N_t-1}),\tau_{N_t}-u(t-
S_{N_t-1})\right)\, du.
\end{equation*}
Then
\begin{equation*}
\begin{split}
 & |\mu_t(f)-\nu_t(f)|  
 \\ 
 & =  \frac{t-S_{N_t-1}}{t}
\left|\int_0^1 \left[f\left(u(t-S_{N_t-1}),\tau_{N_t}-u(t-S_{N_t-1})\right)-f(+\infty,+\infty)\right] du\right|.
\end{split}
\end{equation*}
Since $f(a,b)\to f(+\infty,+\infty)$ as $(a,b)\to (+\infty,+\infty)$ and $f$ is bounded, then the function
\begin{equation*}
\zeta(s):=\int_0^1 \sup_{\tau\geq s}\left|f\left(us,\tau-us)\right)-f(+\infty,+\infty)\right|\, du
\end{equation*}
is bounded, monotone non-increasing and tends to 0 as $s\to+\infty$. Then
\begin{equation*}
\begin{split}
& \{|\mu_t(f)-\nu_t(f)|>\delta\}\subset\left\{\frac{t-S_{N_t-1}}{t}\,\zeta(t-S_{N_t-1})>\delta\right\}=\left\{x_t\,\zeta(tx_t)>\delta\right\}
\end{split}
 \end{equation*}
 where $x_t=\frac{t-S_{N_t-1}}{t}\in[0,1]$. If $x\in[0,1]$ satisfies $x\,\zeta(tx)>\delta$, then $\delta<\zeta(tx)$ and $x\leq \zeta^{-1}(\delta)/t$, so 
that $\delta<C_\delta/t$ and this is impossible as soon as $t\geq C_\delta/\delta$. Therefore, for $t$ large enough the event $\{|\mu_t(f)-
\nu_t(f)|>\delta\}$ is empty.
\end{proof}

\subsection{Free energy}
Recall the definition of $\Gamma$ and \eqref{calpha}-\eqref{e:fsmall}.

\begin{prop}
\label{mgf}
For all $f\in\Gamma$
\begin{equation}
\label{freeener}
\sup_{t>0} \, \E\, e^{t\mu_t(f)} =
\sup_{t>0} \, \E\,\exp\left(\int_0^tf( A_s,B_s)\, ds\right)\leq \frac{D_f}{1-C_f} <+\infty.
%\lim_{t\to+\infty} \frac1t \, \log\E\left(t\langle f,\mu_t\rangle\right)
%=\eps_f.
\end{equation}
\end{prop}
\begin{proof} Since $C_f\in\,]0,+\infty[$, we can introduce the
probability measure
\begin{equation*}
\psi_f(d\tau):=\frac1{C_f}\, \psi(d\tau)\, e^{\tau\overline f(1,\tau)} 
\end{equation*}
and denote by $\zeta_n$ the law of
$S_n$ if $(\tau_i)_{i\in\N^*}$ is i.i.d.\ with common law ${\psi_f}$.
Recalling \eqref{overf} and \eqref{1}
\begin{equation*}
\begin{split}
& \E\exp\left(\int_0^tf( A_s,B_s)\, ds\right)=
\E\left(\un{(N_t=1)}\exp\left( {\tau_1}\, \overline f(t/{\tau_1},{\tau_1})\right) \right)
\\ & +
\sum_{n=1}^{\infty}
\E\left(\un{(N_t=n+1)}\exp\left(\sum_{i=1}^n{\tau_i}\, \overline f(1,\tau_i)
+ \tau_{n+1} \, \overline f\left(\frac{t-S_{n}}{\tau_{n+1}},\tau_{n+1}\right)\right)\right)
\\ & = \int_{]t,+\infty[} \psi(d\tau)\, e^{\tau\,\overline f(t/\tau,\tau)}
+\sum_{n=1}^{\infty}\int_{[0,t]} C_f^n\, \zeta_n(ds) \int_{]t-s,+\infty[} \psi(d\tau) \, e^{\tau\, \overline f((t-s)/\tau,\tau)}
\\ & \leq D_f \sum_{n=0}^{\infty} C_f^n = \frac{D_f}{1-C_f}.
\end{split}
\end{equation*}
\end{proof}

\begin{proof}[Proof of Theorem~\ref{t:ld1}, upper bound]
For $M>0$, $g\in C_b(]0,+\infty]^2)$ and $\delta>0$, let 
\begin{equation*}
\Delta_{M,g,\delta}=\big\{ \mu \in\cP(]0,+\infty]^2)\,: \: \exists\nu\in \Delta, \, |\mu(g)-\nu(g)|\leq\delta, \, \mu(1/(a+b))\le M\big\}
\end{equation*}
and
\begin{equation*}
  R_{M,g,\delta}:= -\varlimsup_{t \to +\infty} \frac{1}{t} \log {\bf P}_t(\Delta_{M,g,\delta}^c).
\end{equation*}
For $\mc A$ measurable subset of $\mc P(]0,+\infty]^2)$ and
for $f \in \Gamma$, by \eqref{freeener},
\begin{equation*}
\begin{split}
    \frac{1}{t} \log {\bf P}_t(\mc A) &
=      \frac{1}{t} \log
        \E\left( e^{t \mu_t(f)}
      e^{-t \mu_t(f)} \un {\mc A}(\mu_t) \right)
\le \frac{1}{t} \log\left[   e^{ -t\inf_{\mu \in \mc A}\mu(f)} \,
\E\left( e^{t \mu_t(f)} \right) \right]
\\ & \leq -\inf_{\mu \in \mc A} \mu(f) +\frac{1}{t} \log\frac{D_f}{1-C_f}
    \end{split}
\end{equation*}
and therefore
  \begin{equation}\label{e:ldmeas}
      \varlimsup_{t \to + \infty} \frac{1}{t} \log \mb P_t(\mc A)
\le  -\inf_{\mu \in \mc A} \mu(f).
\end{equation}
Let now $\mc O$ be
an open subset of $\mc P(]0,+\infty]^2)$. Then applying
\eqref{e:ldmeas} for $\mc A=\mc O \cap \Delta_{M,g,\delta}$
  \begin{equation*}
    \begin{split}
&    \varlimsup_{t\to+\infty} \frac{1}{t} \log {\bf P}_t(\mc O) \le
      \varlimsup_{t\to+\infty} \frac{1}{t} \log\big[2 \max({\bf P}_t
      (\mc O \cap \Delta_{M,g,\delta}),{\bf P}_t(\Delta_{M,g,\delta}^c) )\big]
\\ & \qquad
\le \max\left(-\inf_{\mu \in \mc O \cap \Delta_{M,g,\delta}} \mu(f), -R_{M,g,\delta}\right)=
-\inf_{\mu \in \mc O \cap \Delta_{M,g,\delta}} \mu(f) \wedge R_{M,g,\delta}
    \end{split}
\end{equation*}
which can be restated as
\begin{equation}\label{e:ldopen}
  \varlimsup_{t\to +\infty}  \frac{1}{t} \log {\bf P}_t(\mc O) \le
- \inf_{\mu \in \mc O} \I_{f,M,g,\delta}(\mu)
\end{equation}
for any open set $\mc O$, $f \in \Gamma$ and $M>0$,
where the functional $\I_{f,M,g,\delta}$ is defined as
\begin{equation*}
  \I_{f,M,g,\delta}(\mu):=
  \begin{cases}
    \mu(f) \wedge R_{M,g,\delta} & \text{if $\mu \in \Delta_{M,g,\delta}$}
\\ +\infty & \text{otherwise}.
  \end{cases}
\end{equation*}
Since $f$ is lower semicontinuous and $\Delta_{M,g,\delta}$ is compact
by Lemma~\ref{l:3.2}, then $\I_{f,M,g,\delta}$ is lower
semicontinuous. By minimizing \eqref{e:ldopen} over $\{f,M,g,\delta\}$ we obtain
\begin{equation*}
\varlimsup_{t\to +\infty}  \frac{1}{t} \log {\bf P}_t(\mc O) \le
- \sup_{f,M,g,\delta}\inf_{\mu \in \mc O} \I_{f,M,g,\delta}(\mu)
\end{equation*}
and by applying the minimax lemma
\cite[Appendix 2.3, Lemma 3.3]{KL}, we get that for all compact set $\mc K$
\begin{equation*}
\varlimsup_{t\to +\infty}  \frac{1}{t} \log {\bf P}_t(\mc K) \le
- \inf_{\mu \in \mc K} \sup_{f,M,g,\delta} \I_{f,M,g,\delta}(\mu)
\end{equation*}
i.e.\ $({\bf P}_t)_{t\geq 0}$ satisfies
a large deviations upper bound on compact sets with speed $t$ and rate
$\tilde{\I}(\mu)$ for $\mu\in\cP(]0,+\infty]^2)$
\begin{equation*}
\tilde{\I}(\mu):=  \sup \{\I_{f,M,g,\delta}(\mu)\,: \: f \in \Gamma, \,M>0, \,
g\in C_b(]0,+\infty]^2),\, \delta>0\}.
\end{equation*}
By Lemma~\ref{closed} we have $\cap_{g,\delta}\Delta_{M,g,\delta}\subset
\Delta$, so that $\tilde{\I}(\mu)=+\infty$
if $\mu\notin\Delta$. By Lemma~\ref{l:3.2} and Lemma~\ref{nut}, 
\begin{equation*}
\lim_{M\to +\infty} R_{M,g,\delta}=+\infty, \qquad \forall \, g\in C_b(]0,+\infty]^2),\, \delta>0.
\end{equation*}
Therefore for all $\mu\in\cP(]0,+\infty]^2)$
\begin{equation*}
\tilde{\I}(\mu)\ge  \sup \{\I_{f}(\mu),\,f \in \Gamma\}
\end{equation*}
where
\begin{equation*}
  \I_{f}(\mu):=
  \begin{cases}
    \mu(f) & \text{if $\mu \in \Delta$}
\\ \\
+\infty & \text{otherwise}
  \end{cases}
\end{equation*}
Thus $\tilde{\I}(\mu) \ge \I(\mu)$ by Lemma~\ref{l:3.4}. Therefore
$({\bf P}_t)_{t\geq 0}$ satisfies a large deviations upper bound with rate $\I$ on
compact sets. By Lemma~\ref{l:3.2} and \cite[Lemma 1.2.18]{demzei},
$({\bf P}_t)_{t\geq 0}$ satisfies the full large deviations upper bound on closed sets.
\end{proof}

\section{Lower bound}
In this section we prove the lower bound \eqref{lower} in Theorem~\ref{t:ld1}.

\subsection{Law of large numbers for $\mu_t$}
%For any probability measure $\mu$ on $]0,+\infty[\times]0,+\infty[$ such that
%$\mu(1/(a+b)):=\int 1/(a+b)\, \mu(da,db)\in\R^*_+$ let us set
%\begin{equation*}
%\tilde \mu(da,db):= \frac 1{\mu(1/(a+b))}\, \frac1{a+b}\, \mu(da,db).
%\end{equation*}
For any $\pi\in\cP(]0,+\infty[)$ with $\pi(1/\tau)\in\, ]0,+\infty[$ we recall
that
\begin{equation*}
\tilde \pi(d\tau):= \frac 1{\pi(1/\tau)}\, \frac1\tau\, \pi(d\tau).
\end{equation*}
and we denote
by $\bbP_{\tilde \pi}$ the law of an i.i.d.\ sequence $(\tau_i)_{i\geq 1}$ with
marginal distribution $\tilde \pi$, 
% and by ${\bf P}_t$ the law of $\mu_t$ when $v$ has distribution $\bbP_{\pi}$, 
i.e.\
\begin{equation}\label{bbqphi}
\bbP_{\tilde \pi}:=\otimes_{i\in\N^*}\tilde \pi(d\tau_i).%, \qquad {\bf P}_t:=\bbP_{\pi}\circ\mu_t^{-1}.
\end{equation}

\begin{prop}\label{limit}
Let $\pi\in\cP(]0,+\infty[)$ with $\pi(1/\tau)\in\,]0,+\infty[$.
Under $\bbP_{\tilde \pi}$,  a.s. 
\begin{equation*}
\mu_t\rightharpoonup \int_{[0,1]\times]0,+\infty[} \delta_{(u\tau,(1-u)\tau)}\, du \otimes \pi(d\tau) \quad\text{on $]0,+\infty]^2$}, \qquad t\to+\infty.
\end{equation*}
\end{prop}
\begin{proof}
For all $f\in C(]0,+\infty]^2)$ we recall the notation
\eqref{overf}
\begin{equation*}
\overline f(r,\tau):=\int_0^r f(u\tau,(1-u)\tau)\, du, \qquad
r\in[0,1], \, \tau>0,
\end{equation*}
and, by \eqref{1}
\begin{equation*}
\mu_t(f) =\frac1t \, \sum_{i=1}^{N_t-1}  {\tau_i} \, \overline f(1,\tau_i)\,
+\frac{\tau_{N_t}}{t}\, \overline f\left(\frac{t-S_{N_t-1}}{\tau_{N_t}},\tau_{N_t} \right).
\end{equation*}
By the strong law of large numbers a.s.
\begin{equation*}
\begin{split}
\lim_{n\to+\infty} \frac1n\sum_{i=1}^{n}  {\tau_i}\, \overline f(1,\tau_i)
=\tilde\pi(\tau\, \overline f(1,\tau)) =\frac1{\pi(1/\tau)}\, \pi(\overline f(1,\tau)).
\end{split}
\end{equation*}
By the renewal Theorem, a.s.
\begin{equation*}
\lim_{t\to+\infty} \frac{N_t-1}t = \frac1{\bbE_{\tilde \pi}(\tau_1)}=\frac{\pi(1/\tau)}{\int \tau\,\frac1\tau\, \pi(d\tau)} =\pi(1/\tau)\in\, ]0,+\infty[.
\end{equation*}
Therefore a.s.
\begin{equation*}
\lim_{t\to+\infty} \frac{{N_t-1}}t \, \frac1{{N_t-1}}\sum_{i=1}^{N_t-1}   {\tau_i} \, \overline f(1,\tau_i)=
\, \pi\left(\overline f(1,\tau)\right).
\end{equation*}
On the other hand, by the law of large numbers a.s.
\begin{equation*}
\lim_{n\to+\infty} \frac{S_n}n=\tilde\pi(\tau)=\frac1{\pi(1/\tau)},
\end{equation*}
so that
a.s.
\begin{equation*}
\lim_{t\to+\infty} \frac{S_{N_t-1}}t=\lim_{t\to+\infty} \frac{S_{N_t-1}}{N_t-1} \, \frac{{N_t-1}}t=1,
\qquad \lim_{t\to+\infty} \frac{t-S_{N_t-1}}t=0.
\end{equation*}
It follows that a.s.
\begin{equation*}
\lim_{t\to+\infty} \left| \frac{\tau_{N_t}}{t}\, \overline f\left(\frac{t-S_{N_t-1}}{\tau_{N_t}},\tau_{N_t} \right) \right|
\leq \lim_{t\to+\infty} \frac{t-S_{N_t-1}}t \, \|f\|_\infty =0.
\end{equation*}
\end{proof}

\subsection{Proof of the lower bound}

For the proof of the lower bound, it is well known that it is enough to show the following
\begin{prop}\label{oraet}
For every $\mu\in\Delta$ there exists a family ${\bf Q}_t$ of
probability measures on $\mc P(]0,+\infty]^2)$ such that
${\bf Q}_t\rightharpoonup\delta_\mu$
and
\begin{equation*}
\varlimsup_{t\to+\infty} \frac1t  \H({\bf Q}_t \,| \, {\bf P}_t)\leq \I(\mu).
\end{equation*}
\end{prop}
Indeed, if Proposition~\ref{oraet} is proved, then we reason as follows.
Let $\mu\in\Delta$ and let $\mc V$ be an open neighborhood of $\mu$
in the weak topology. Then
\begin{equation*}
\begin{split}
\log {\bf P}_t(\mc V) & =\log \int_\mc V \frac{d{\bf P}_t}
{d{\bf Q}_t} \, {d{\bf Q}_t} = \log \left( \frac1{{\bf Q}_t(\mc V)}\int_\mc V
\frac{d{\bf P}_t} {d{\bf Q}_t} \, {d{\bf Q}_t}\right) + \log {\bf Q}_t(\mc V)
\\ &\geq \frac1{{\bf Q}_t(\mc V)}\int_\mc V
\log \left( \frac{d{\bf P}_t} {d{\bf Q}_t}\right)  \, {d{\bf Q}_t}+ \log {\bf Q}_t(\mc V)
\end{split}
\end{equation*}
by using Jensen's inequality. Now, since $x\log x\geq -e^{-1}$ for all $x\geq 0$,
we obtain
\begin{equation*}
\begin{split}
\log {\bf P}_t(\mc V) &\geq \frac1{{\bf Q}_t(\mc V)}\left( -\H({\bf Q}_t\,|\,{\bf P}_t)+\int_{\mc V^c}
\log \left( \frac{d{\bf Q}_t} {d{\bf P}_t}\right)  \, \frac{d{\bf Q}_t} {d{\bf P}_t}\, {d{\bf P}_t}\right)+ \log {\bf Q}_t(\mc V)
\\ &\geq \frac1{{\bf Q}_t(\mc V)}\left( -\H({\bf Q}_t\,|\,{\bf P}_t)-e^{-1}
\right)+ \log {\bf Q}_t(\mc V).
\end{split}
\end{equation*}
Since $\mu\in \mc V$, ${\bf Q}_t\rightharpoonup\delta_\mu$ and $\mc V$ is open, then 
${\bf Q}_t(\mc V)\to 1$ as $t\to+\infty$. We obtain
\begin{equation*}
\varliminf_{t\to+\infty} \frac1t \, \log {\bf P}_t(\mc V)\geq 
-\varlimsup_{t\to+\infty} \frac1t  \H({\bf Q}_t \,| \, {\bf P}_t)\geq -\I(\mu).
\end{equation*}
Therefore, for any open set $\mc O$ and for any $\mu\in \mc O$
\begin{equation*}
\varliminf_{t\to+\infty} \frac1t \, \log {\bf P}_t(\mc O)\geq-\I(\mu),
\end{equation*}
and by optimizing over $\mu\in \mc O$ we have the lower bound.
\begin{proof}[Proof of Proposition~\ref{oraet}] Let us first suppose that $\mu\in\Delta_0$ as in \eqref{e:muomega0}.
Notice that $\mu(1/\tau)=\pi(1/\tau)\in\, ]0,+\infty[$.
Fix $\delta>0$ and
set $T_t:= \lfloor\pi(1/\tau)\, (1+\delta)\, t\rfloor$. For
$t>1/\pi(1/\tau)$, let us denote by $\bbP^{t,\delta}$ the law on $]0,+\infty]^{\N^*}$ such that under
$\bbP^{t,\delta}$ the sequence $(\tau_i)_{\tau\geq 1}$ is independent and
\begin{enumerate}
\item for all $i\leq T_t$, $\tau_i$ has law $\tilde\pi$
\item for all $i\geq T_t+1$, $\tau_i$ has law $\psi$.
\end{enumerate}
Let us set ${\bf Q}_{t,\delta}:=\bbP^{t,\delta}\circ\mu_t^{-1}$. Let us prove now that
\begin{equation}\label{sieg}
\lim_{\delta\downarrow 0}\lim_{t\uparrow +\infty} {\bf Q}_{t,\delta} = \delta_\mu.
\end{equation}
By the law of large numbers of Proposition~\ref{limit}, under $\bbP_{\tilde\pi}$ we have a.s.
\begin{equation*}
\lim_{t\to+\infty} \frac{S_{T_t}}t =
\lim_{t\to+\infty} \frac{S_{T_t}}{T_t} \,\frac{T_t}t=
\frac1{\pi(1/\tau)}\, {\pi(1/\tau)}\, (1+\delta)=1+\delta.
\end{equation*}
However $S_{T_t}$
has the same law under $\bbP_{\tilde\pi}$ and under $\bbP^{t,\delta}$, so we obtain
for any $\delta>0$
\begin{equation}\label{mire}
\lim_{t\to+\infty} \bbP^{t,\delta}\left(S_{T_t}\leq t\right)= 
\lim_{t\to+\infty} \bbP_{\tilde\pi}\left(\frac{S_{T_t}}t\leq 1\right)=0.
\end{equation}
Therefore, if we set
\begin{equation*}
D_{t,\delta}:=\left\{S_{T_t}>t\right\}
\end{equation*}
then, by \eqref{mire} we obtain that
for all $\delta>0$
\begin{equation}
\label{e:aa}
\lim_{t\to+\infty} \bbP^{t,\delta}\left(D_{t,\delta} \right)=1.
\end{equation}
We recall that $\{S_n>t\}=\{N_t\leq n\}$. 
Therefore on $D_{t,\delta}$ we have $N_t\leq T_t$ and therefore by \eqref{1} for any $f\in C_b(]0,+\infty]^2)$
\begin{equation*}
\bbP^{t,\delta}(|\mu_t(f)-\mu(f)|>\eps)\leq\bbP_{\tilde\pi}(\{|\mu_t(f)-\mu(f)|>\eps\}\cap D_{t,\delta})+
\bbP^{t,\delta}(D_{t,\delta}^c)
\end{equation*}
By Proposition~\ref{limit}
\begin{equation*}
\lim_{\delta\downarrow 0}\lim_{t\uparrow+\infty} \bbP_{\tilde\pi}\left(\{|\mu_t(f)-\mu(f)|>\eps\}\cap D_{t,\delta} \right)=0,
\end{equation*}
which, in view of \eqref{e:aa}, implies \eqref{sieg}.

Now we estimate the entropy
\begin{equation}
\label{e:Hbound}
 \begin{split}
 \H({\bf Q}_{t,\delta} \,| \, {\bf P}_t)  \le 
 \H\left(\bbP^{t,\delta} \, |\, \bbP_{\psi}\right)
= T_t \H(\tilde\pi \, | \, \psi),
\end{split}
\end{equation}
so that
\begin{equation*}
\lim_{\delta \downarrow 0} \varlimsup_{t\uparrow +\infty}
\frac 1t \H({\bf Q}_{t,\delta} \,| \, {\bf P}_t)  \le  \pi(1/\tau) \H(\tilde\pi \, | \, \psi).
\end{equation*}
Then there exists a map $t\mapsto \delta(t)>0$ vanishing as $t \uparrow +\infty$ such that
${\bf Q}_t:= {\bf Q}_{t,\delta(t)} \to \delta_\mu$ and $\varlimsup_t t^{-1}
\H({\bf Q}_t \,| \, {\bf P}_t) \le \I(\mu)$.

Let now $\mu\in\Delta\setminus\Delta_0$. Then, by Proposition~\ref{p:igood}-(3) (applied with $\psi_n=\psi$)
we can find a sequence $(\mu_n)_n$ in ${\Delta}_0$
such that $\mu_n\rightharpoonup\mu$ and $\varlimsup_n \I(\mu_n) \le \I(\mu)$. Moreover, we now know that there exists for all $n$
 a family ${\bf Q}^n_t$ of
probability measures on $\mc P(]0,+\infty]^2)$ such that
${\bf Q}_t^n\rightharpoonup\delta_{\mu_n}$
and
\begin{equation*}
\varlimsup_{t\to+\infty} \frac1t \, \H({\bf Q}_t^n \,| \, {\bf P}_t)\leq \I(\mu_n).
\end{equation*}
With a standard diagonal procedure we can find a family ${\bf Q}_t$ such that
${\bf Q}_t\rightharpoonup\delta_\mu$
and
\begin{equation*}
\varlimsup_{t\to+\infty} \frac1t \, \H({\bf Q}_t \,| \, {\bf P}_t)\leq \I(\mu).
\end{equation*}
\end{proof}

\section{Large deviations of $C_t/t$}
In this section we prove Theorem~\ref{t:ldcounting},
with $F \colon ]0,+\infty[ \to[0,+\infty[$ continuous and bounded, and we
set ${\tilde F}(\tau):=F(\tau)/\tau$, $\tau\in\,]0,+\infty]$.
%and satisfying \eqref{assump}, i.e.\
%\begin{equation*}
%\text{$F$ and ${\tilde F}$ continuous and bounded}, \qquad
%F(\tau):={{\tilde F}(\tau)}\, \tau, \qquad \forall \, \tau\in\, ]0,+\infty[.
%\end{equation*}
We remark that $A_t+B_t=\tau_{N_t}$ and we define the empirical measure
$\nu_t$ of $(\tau_{N_s})_{s\geq 0}$
\begin{equation*}
\nu_t(O):= \frac 1t \int_{[0,t[}\un{O}(\tau_{N_s})\, ds=\int \un{O}(a+b)\, \mu_t(da,db), \qquad O\subset\,]0,+\infty[.
\end{equation*}
Notice that by \eqref{1}
\begin{equation}\label{1bis}
\nu_t(O) = \frac1t \sum_{i=1}^{N_t-1}  {\tau_i}\, \un{O}(\tau_i)
+ \frac{t-S_{N_t-1}}{t}\, \un{O}(\tau_{N_t}).
\end{equation}
Then
\begin{equation*}
\nu_t({\tilde F})=\nu_t(F/\tau)= \frac1t\sum_{i=1}^{N_t-1} F(\tau_i) +
\frac{1}{t}\frac{t-S_{N_t-1}}{\tau_{N_t}}\, F(\tau_{N_t}), \qquad t>0,
\end{equation*}
by the representation \eqref{1}. So that a.s.
\begin{equation}\label{r:picounting}
\left|\nu_t({\tilde F})-\frac1t\sum_{i=1}^{N_t-1} F(\tau_i)\right| \le \frac{\|F\|_\infty}t.
\end{equation}
In particular, $\frac1t \sum_{i=1}^{N_t-1} F(\tau_i)$ and $\nu_t({\tilde F})$
are {\it exponentially equivalent}, i.e.\
\begin{equation*}
\varlimsup_{t\to+\infty} \frac1t \log \bbP\left(\left|\nu_t({\tilde F})-\frac1t\sum_{i=1}^{N_t-1} F(\tau_i)\right|>\delta\right)=-\infty, \qquad \forall \, 
\delta>0.
\end{equation*}
By \cite[Theorem 4.2.13]{demzei}, if the law of $\nu_t({\tilde F})$ satisfies a large deviations principle, the same large deviation principle holds 
for the law of $\frac1t\sum_{i=1}^{N_t-1} F(\tau_i)$. Moreover we have $\nu_t({\tilde F})=\mu_t(G)$ where
\begin{equation*}
G \colon ]0,+\infty]^2 \to [0,+\infty[, \qquad 
G(a,b):={\tilde F}(a+b).
\end{equation*}
This suggests to derive
large deviations for $\frac1t\sum_{i=1}^{N_t-1} F(\tau_i)$ by using the classical contraction principle \cite[Theorem 4.2.1]{demzei} over the 
map
$\mc P(]0,+\infty]^2) \ni \mu \mapsto \mu(G) \in [0,+\infty[$. We shall start off by computing the candidate rate functional, then we
consider the case of ${\tilde F}$ bounded, and finally we show how to remove this assumption.

\subsection{The case of a bounded $\tilde F$}
In the above setting, set
\begin{equation}\label{JJ}
J(m):=\inf \big\{\I(\mu)\,:\: \mu \in \mc P(]0,+\infty]^2),\, 
                                \mu(G)=m  \big\}, \qquad m\in[0,+\infty].
\end{equation}
We compute now this rate functional.
\begin{lemma}
\label{r:IN}
Recall \eqref{Lambda_F} and \eqref{e:IN}. For each $m\in
[0,+\infty]$ we have $J(m)=J_{F}(m)$. Moreover the $\inf$ in \eqref{JJ} is attained for all finite $m$.
%\begin{equation*} 
%J(m)=J_{F}(m).
%= \inf \big\{\I(\mu)\,:\: \mu \in \mc P(]0,+\infty]^2),\,      \mu({\tilde F})=m  \big\}.
%\end{equation*}
\end{lemma}
\begin{proof} By \eqref{e:IN}
\begin{equation*}
\begin{split}
& \inf \big\{\I(\mu)\,:\: \mu \in \Delta,\, \mu(G)=m  \big\}=
 \inf \big\{\I(\mu)\,:\: \mu \ \text{as in \eqref{e:muomega0}-\eqref{e:muomega}}, \, \alpha\pi({\tilde F})=m  \big\}
\\ & =
 \inf \big\{\I(\mu)\,:\: \mu \ \text{as in \eqref{e:muomega0}-\eqref{e:muomega}}, \, \tilde\pi(\tau)=\beta, \, \tilde\pi(\tau {\tilde F})=\beta m/\alpha, 
\beta>0  \big\}
\\ &  =\inf \big\{(\alpha/\beta) \H(\tilde\pi \, | \, \psi)+(1-\alpha)\, \xi \,:\:
\tilde\pi(\tau)=\beta,\,\tilde\pi(F)=\beta m/\alpha, \alpha\in[0,1], \beta>0  
\big\},
\end{split}
\end{equation*}
where we have used that by \eqref{e:tildepi/tau}
\begin{equation*}
\pi({\tilde F})=\frac{\tilde\pi(\tau {\tilde F})}{\tilde\pi(\tau)}=
\frac{\tilde\pi(F)}{\tilde\pi(\tau)}, \qquad \pi(1/\tau)=\frac{1}{\tilde\pi(\tau)}.
\end{equation*}
Now, setting
\begin{equation*}
p(a,b):=\inf\{ \H(\zeta \, | \, \psi)\,:\:
\zeta(\tau)=a,\,\zeta(F)=b  \big\}, 
\end{equation*}
then $p\equiv\Lambda^*$ by \cite[Theorem 3]{csiszar}, where, in the notation \eqref{Lambda_F}, $\Lambda(x,y)=\log \psi(e^{x\tau+y F})$ and 
$\Lambda^*$ is the Legendre transform of $\Lambda$. Another way to check that $p\equiv\Lambda^*$ is the following: $p$ and $\Lambda^*$ 
are easily seen to be lower semicontinuous convex functions of $(a,b)$ and moreover the Legendre transform of $p$ is
\begin{equation*}
\begin{split}
p^*(x,y) & =\sup_{a,b}\left( ax+by-p(a,b)\right)=\sup_{a,b,\zeta}\left\{
ax+by- \H(\zeta \, | \, \psi)\,:\: \zeta(\tau)=a,\,\zeta( F)=b \right\}
\\ & =\sup_\zeta\left\{\zeta(x\tau+y F)- \H(\zeta \, | \, \psi) \right\}
=\log \psi(e^{x\tau+y F})=\Lambda(x,y),
\end{split}
\end{equation*}
so that $p=\Lambda^*$. Therefore
\begin{equation*}
\begin{split}
J(m) & = \inf \big\{\I(\mu),\,\mu \in \Delta,\,:\: \mu({\tilde F})=m  \big\} 
\\ & =\inf\left\{ (\alpha/\beta)\Lambda^*(\beta,\beta m/\alpha)+(1-\alpha)\, \xi,\, \alpha\in[0,1], \, \beta>0\right\}
\\ & =\inf\left\{ \beta\Lambda^*(\alpha/\beta, m/\beta)+(1-\alpha)\, \xi,\, \alpha\in[0,1], \, \beta>0\right\}.
\end{split}
\end{equation*}
We want now to prove that $J(m)=J_F(m)$, recall  \eqref{e:IN}. In particular we show that for all $\beta>0$
\begin{equation}
\label{morand}
\inf_{\alpha\in[0,1]}\left\{ \beta\Lambda^*(\alpha/\beta, m/\beta)+(1-\alpha)\, \xi \right\} = \beta\Lambda^*(1/\beta, m/\beta).
\end{equation}
First notice that the left hand side of \eqref{morand} is clearly less or equal to the right hand side by choosing $\alpha=1$. We now prove the 
converse inequality.
For all $\alpha\in[0,1]$
\begin{equation*}
\begin{split}
\beta\Lambda^*(\alpha/\beta, m/\beta)+(1-\alpha)\, \xi & =  \sup_{x,y} \left(\alpha  x+(1-\alpha)\xi+my-\beta\Lambda(x,y) \right)
\\ &
 \geq \sup_{x,y}  \left(x\wedge\xi+my-\beta\Lambda(x,y) \right).
\end{split}
\end{equation*}
Now, since $F$ is bounded, then $\Lambda(x,y)=+\infty$ for all $x>\xi$, so that
the supremum over $x$ can be restricted to a supremum over $\{x\leq \xi\}$. Therefore we obtain 
\begin{equation*}
\beta\Lambda^*(\alpha/\beta, m/\beta)+(1-\alpha)\, \xi \geq \sup_{x,y}  \left(x+my-\beta\Lambda(x,y) \right) = \beta\Lambda^*(1/\beta, m/\beta)
\end{equation*}
and \eqref{morand} is proven. 

Finally, in order to prove that the $\inf$ in \eqref{JJ} is attained, let us use the formula obtained at the beginning of the proof
\begin{equation*}
\begin{split}
& \inf \big\{\I(\mu), \mu \in \Delta\,:\: \mu(G)=m  \big\}=
\\ &
\quad   =\inf \big\{(\alpha/\beta) \H(\tilde\pi \, | \, \psi)+(1-\alpha)\, \xi, \,
\tilde\pi(\tau)=\beta,\,\tilde\pi(F)=\beta m/\alpha, \alpha\in[0,1], \beta>0  
\big\}.
\end{split}
\end{equation*}
We consider a minimizing sequence $(\alpha_n,\tilde\pi_n,\beta_n)$ and the associated $\mu_n\in\Delta$. use coercivity and lower semi-
continuity of the relative entropy and the bound $|\beta|\leq \|F\|_\infty/m$, and extract a sequence converging to $(\alpha,\zeta,\beta)$. Now 
we have to prove that the limit still satisfies the required constraint, in particular that $\zeta(\tau)=\beta$, since the rest follows easily. Let us 
notice that for all $\delta>0$
\begin{equation*}
\int_{]0,\delta]} \frac1\tau \, \pi_n(d\tau) = \frac{\tilde\pi_n(]0,\delta])}{\tilde\pi_n(\tau)}=\frac{\tilde\pi_n(]0,\delta])}\beta
\end{equation*}
and since $(\tilde\pi_n)$ is tight in $]0,+\infty[$ we obtain
\begin{equation*}
\lim_{\delta\to 0} \sup_n \int_{]0,\delta]} \frac1\tau \, \pi_n(d\tau) = 0.
\end{equation*}
It follows that $(\pi_n)_n$ is tight in $]0,+\infty]$; if $\pi_{n_k}\rightharpoonup\pi$ in $\cP(]0,+\infty])$, by a uniform integrability argument, we 
obtain that $\pi_n(1/\tau)\to\pi(1/\tau)$ and that $\zeta=\tilde\pi$, i.e.\ in particular $\pi_n\rightharpoonup\pi$. Since $\pi_n(1/\tau)=1/\tilde
\pi_n(\tau)=1/\beta$, we obtain that $\zeta(\tau)=\beta$ and the $\inf$ above is attained, so that we can reconstruct $\mu\in\Delta$ attaining 
the minimum in \eqref{JJ}.
\end{proof}

If ${\tilde F}$ is bounded, then $G$ is bounded too and we have the following
\begin{remark}\label{contrac}
The map $\mc P(]0,+\infty]^2) \ni \mu \mapsto \mu(G) \in [0,+\infty[$ is
continuous in the weak topology.
\end{remark}
\begin{proof}
Notice that ${\tilde F}$ is bounded and continuous on $]0,+\infty[$ and ${\tilde F}(\tau)=F(\tau)/\tau\to 0$ as $\tau\to+\infty$, so that it has a 
unique continuous extension to $]0,+\infty]$.
Then the map $G$ defined above
is bounded and continuous and thus $\mu \mapsto \mu(G)$ is continuous.
\end{proof}
By the contraction principle \cite[Theorem 4.2.1]{demzei}, we obtain that
the law of $\mu_t(G)$ satisfies a large deviations principle with speed $t$ and rate functional $J$ given by \eqref{JJ}, which is equal to $J_F$ 
by Lemma~\ref{r:IN}.

\subsection{The case of general $\tilde F$}
Now we remove the assumption that ${\tilde F}$ be bounded, always assuming $F$ to be bounded and continuous. In this case, the map $
\nu\mapsto \nu({\tilde F})$ is no more necessarily continuous as in Remark~\ref{contrac}.

We introduce now the approximation that will allow us to justify the
use of the classical contraction principle.
We fix $\eps>0$ and we define the processes
\begin{equation*}
S^\eps_n:=\sum_{i=1}^n \tau_i\vee \eps, \quad n\geq 1,
\qquad N^\eps_t:=\#\{n\geq 0\,:\: S^\eps_n\leq t \}=
\inf\left\{ n\geq 0\,:\: S_{n}^\eps >t\right\},
\end{equation*}
and for all $t\geq 0$
\begin{equation*}
A^\eps_t:=t-S^\eps_{N^\eps_t-1}, \qquad B^\eps_t:=S^\eps_{N^\eps_t}-t.
\end{equation*}
Define the empirical measure
\begin{equation*}
\mu^\eps_t:= \frac 1t \int_{[0,t[}\delta_{(A^\eps_s,B^\eps_s)}\, ds
\in \mc P(]0,+\infty]^2)
\end{equation*}
and denote by ${\bf P}^\eps_t$ the law of $\mu^\eps_t$. Notice that
$(S^\eps_n,N^\eps_t,A^\eps_t,B^\eps_t,\mu^\eps_t)$ under $\bbP$
have the same law as $(S_n,N_t,A_t,B_t,\mu_t)$
under $\bbP_{\psi^\eps}$ (recall \eqref{bbqphi}), where
\begin{equation*}
\psi^\eps(d\tau):= \psi(]0,\eps])\, \delta_\eps(d\tau)+\un{(\tau>\eps)}\psi(d\tau).
\end{equation*}
We denote by $\Lambda^\eps$, $\xi^\eps$ and $J_F^\eps$ the quantities defined by \eqref{Lambda_F}, \eqref{e:xiphi} and \eqref{e:IN} 
replacing $\psi$ by $\psi^\eps$ and remark that in fact $\xi=\xi^\eps$.
Then we have the following
\begin{lemma}\label{lige}
 The law of the random variable $\frac1t\sum_{i=1}^{N^\eps_t-1} F(\tau_i^\eps)$ satisfies a large deviations principle with rate $J^\eps_{F}$.
\end{lemma}
\begin{proof}
By Theorem~\ref{t:ld1}, ${\bf P}^\eps_t$ satisfies a large deviations principle with good rate functional
\begin{equation}\label{I^gep}
\I^\eps=\begin{cases}
\alpha\,\pi(1/\tau) \H\big(\tilde\pi\, \big|\, \psi^\eps\big)   + (1-\alpha)\, \xi
           & \text{if $\mu \in \Delta$ is given by \eqref{e:muomega0}-\eqref{e:muomega} }
\\
+\infty & \text{if $\mu \notin \Delta$}.
  \end{cases}
\end{equation}
For each
  $\eps>0$, the map $\mc P(]0,+\infty]^2) \ni \nu \mapsto \nu(F/(\tau
  \vee \eps)) \in [0,\eps^{-1}]$ is continuous. Since $\mu_t(F/\tau) =
  \mu_t(F/(\tau \vee \eps))$ almost surely under $\bbP_{\psi^\eps}$, 
  Lemma~\ref{r:IN} and the classical
  contraction principle imply that the law of $\mu_t^\eps(F/\tau)$ satisfies a large
 deviations principle with speed $(t)$ and rate $J^\eps_{F}$.
  By \eqref{r:picounting}, $\mu^\eps_t(F/\tau)$ and $\frac1t\sum_{i=1}^{N^\eps_t-1} F(\tau_i^\eps)$ are
  exponentially close, so that by \cite[Theorem 4.2.13]{demzei} we obtain
  the desired result.
\end{proof}

The following lemma states that $(N_t^\eps/t)_{t>0}$ is an exponentially good approximation of $(N_t/t)_{t>0}$.
\begin{lemma}\label{close}
For all $\delta>0$
\begin{equation*}
\lim_{\eps\downarrow0} \varlimsup_{t\to+\infty} \frac1t \log \bbP(|N_t-N^\eps_t|
>t\delta)=-\infty.
\end{equation*}
\end{lemma}
\begin{proof}
Notice that $N_t\geq N^\eps_t$. For $\delta>0$ and $M>0$
\begin{equation*}
\begin{split}
& \bbP(N_t-N^\eps_t>t\delta)\leq \sum_{n=0}^{\lfloor Mt\rfloor} 
\bbP(N_t-N^\eps_t>t\delta, \, N_t=n) + \bbP(N_t>Mt)
\\ & = \sum_{n=0}^{\lfloor Mt\rfloor} \sum_{k=0}^n
\bbP\left(N_t-N^\eps_t>t\delta, \, N_t=n, \, \sum_{i=1}^n\un{(\tau_i<\eps)}=k\right) + \bbP(S_{\lfloor Mt\rfloor}\leq t).
\end{split}
\end{equation*}
Now, for $k\leq n\leq Mt$ and $m\leq n$, on the event $\left\{N_t=n, \, N^\eps_t=m, \, \sum_{i=1}^n\un{(\tau_i<\eps)}=k\right\}$ we have that 
\begin{equation*}
t<S^\eps_{m} \leq S_{m}+k\eps \, \Longrightarrow \, S_{m}>t-k\eps \, \Longrightarrow \, N_{t-k\eps}\leq m=N^\eps_t
\end{equation*}
and finally $N_{t-Mt\eps}\leq N^\eps_t$.
Therefore
\begin{equation*}
 \bbP(N_t-N^\eps_t>t\delta)\leq \frac{(Mt)^2}2
\bbP\left(N_t-N_{t-Mt\eps}>t\delta\right) + \bbP(S_{\lfloor Mt\rfloor}\leq t).
\end{equation*}
Now, we can write for $s<t$ and $k\in\N$
\begin{equation*}
\{N_t-N_{s}> k\}=\{S_{N_s+k}\leq t\}=\{S_{N_s}+\hat S_k\leq t\}
\subset \{\hat S_k\leq t-s\}
\end{equation*}
where $(\hat S_k:=S_{N_s+k}-S_{N_s}, \, k\geq 1)$, has the same
law as $(S_k, \, k\geq 1)$. Then
\begin{equation*}
 \bbP(N_t-N^\eps_t>t\delta)\leq \frac{(Mt)^2}2
\bbP\left(S_{\lfloor t\delta\rfloor}\leq Mt\eps\right) + \bbP(S_{\lfloor Mt\rfloor}\leq t).
\end{equation*}
Therefore 
\begin{equation*}
\begin{split}
& \lim_{\eps\downarrow0} \varlimsup_{t\to+\infty} \frac1t \log \bbP(N_t-N^\eps_t>t\delta)
\\ & \leq \varlimsup_{M\to+\infty}\lim_{\eps\downarrow0} \varlimsup_{t\to+\infty} \frac1t \log \left[2\max\left\{\frac{(Mt)^2}2
\bbP\left(S_{\lfloor t\delta\rfloor}\leq Mt\eps\right), \, \bbP(S_{\lfloor Mt\rfloor}\leq t)\right\}\right]
\\ & \leq  \max\left\{\lim_{\eps\downarrow0} \varlimsup_{t\to+\infty} \frac1t \log
\bbP\left(S_{\lfloor t\delta\rfloor}\leq t\eps\right), \, \varlimsup_{M\to+\infty}\varlimsup_{t\to+\infty} \frac1t \log\bbP(S_{\lfloor Mt\rfloor}\leq t)\right\}.
\end{split}
\end{equation*}
Arguing as in the proof of Lemma~\ref{l:3.2}, by the Markov inequality
\begin{equation*}
\bbP\left(S_{\lfloor t\delta\rfloor}\leq t\eps\right)=
\bbP\left(e^{-S_{\lfloor t\delta\rfloor}/\eps}\geq e^{-t}\right)\leq e^{t+\lfloor t\delta\rfloor\log\E(e^{-\tau_1/\eps})}
\end{equation*}
so that
\begin{equation*}
\lim_{\eps\downarrow0} \varlimsup_{t\to+\infty} \frac1t \log
\bbP\left(S_{\lfloor t\delta\rfloor}\leq t\eps\right)\leq
\lim_{\eps\downarrow0}\left(1+\delta\log\E(e^{-\tau_1/\eps})\right)=-\infty,
\end{equation*}
and analogously
\begin{equation*}
\begin{split}
\varlimsup_{M\to+\infty}\varlimsup_{t\to+\infty} \frac1t \log\bbP(S_{\lfloor Mt\rfloor}\leq t)
\leq \varlimsup_{M\to+\infty}\varlimsup_{t\to+\infty} \frac1t \log e^{t+(\lfloor Mt\rfloor)\log \E(e^{-\tau_1})}=-\infty.
\end{split}
\end{equation*}
\end{proof}
Let us define
\begin{equation*}
C^\eps_t:=\sum_{i=1}^{N_t^\eps-1} F(\tau_i^\eps), \qquad t>0.
\end{equation*}
We deduce from Lemma~\ref{close} that 
\begin{lemma}\label{close2}
The process $(C_t^\eps/t)_{t>0}$ is an exponentially good approximation of the process $(C_t/t)_{t>0}$, i.e.\ for all $\delta>0$
\begin{equation}\label{demz11}
\lim_{\eps\downarrow0} \varlimsup_{t\to+\infty} \frac1t \log \bbP\left(\left|\frac1t\sum_{i=1}^{N_t^\eps-1} F(\tau_i^\eps)-\frac1t\sum_{i=1}^{N_t-1} 
F(\tau_i)\right|>\delta\right)=-\infty.
\end{equation}
\end{lemma}
\begin{proof} Let $\omega_F(\eps):=\sup\{|F(\eps)-F(\eta)|,\,\eta\in[0,\eps]\}$.
Since $N_t\geq N_t^\eps$,
\begin{equation*}
\begin{split} 
\left|\frac1t\sum_{i=1}^{N_t^\eps-1} F(\tau_i^\eps)-\frac1t\sum_{i=1}^{N_t-1} F(\tau_i)\right| 
& \leq \frac1t\sum_{i=1}^{N_t^\eps-1} \left|F(\tau_i^\eps)-
F(\tau_i)\right| + \frac1t\sum_{i=N_t^\eps-1}^{N_t-1} \left|
F(\tau_i)\right|
\\ & 
\leq \frac{N_t-1}t \, \omega_F(\eps)+ \frac{N_t-N_t^\eps}t\, \|F\|_\infty.
\end{split}
\end{equation*}
Since $F$ is continuous in $0$, $\omega_F(\eps)\to 0$ as $\eps\downarrow 0$ 
and by Lemmas~\ref{l:3.2}-\ref{close} we conclude.
\end{proof}

Since $(C_t^\eps/t)_{t>0}$ is an exponentially good approximation of the process $(C_t/t)_{t>0}$
by Lemma~\ref{close2}, then by Lemma~\ref{lige} and \cite[Theorem 4.2.16]{demzei} we have that $(C_t/t)$ satisfies a large deviations 
principle with rate 
\begin{equation*}
\tilde J(m):=\sup_{\delta>0} \,\varliminf_{\eps\downarrow 0} \, \inf_{z\,:\: |z-m|\leq \delta} \, J^\eps_F(z).
\end{equation*}

\begin{proof}[Proof of Theorem~\ref{t:ldcounting}]
By Lemma~\ref{close2}, $(C_t^\eps/t)_{t>0}$ is an exponentially good approximation of the process $(C_t/t)_{t>0}$, then by Lemma~\ref{lige} 
and \cite[Theorem 4.2.16]{demzei} we have that $(C_t/t)$ satisfies a large deviations principle with rate
\begin{equation*}
\sup_{\delta>0} \,\varliminf_{\eps\downarrow 0} \, \inf_{z\,:\: |z-m|\leq \delta} \, J^\eps_F(z).
\end{equation*}
which equals $J_F$ as a straightforward consequence of Proposition~\ref{p:igood}. Remark that we have also proved \eqref{IJ_F} and, still 
by 
Proposition~\ref{p:igood}-(3), \eqref{I_0J_F}.
\end{proof}

\end{document}